\documentclass[11 pt, leqno]{amsart}

\usepackage{amsfonts}
\usepackage{amsthm}
\usepackage{amssymb}
\usepackage{latexsym}
\usepackage{amsmath}
\usepackage{bbm}
\usepackage{dsfont}
\usepackage{upgreek}
\usepackage{color}

\usepackage{a4wide}
\usepackage{hyperref}
\usepackage{comment}

\usepackage{graphicx}

\usepackage{amstext}

\theoremstyle{plain}
\newtheorem{tw}{Theorem}[section]
\newtheorem{quest}[tw]{Question}

\newtheorem {lem} [tw]{Lemma}
\newtheorem {prop}[tw] {Proposition}

\newtheorem{cor}[tw]{Corollary}

\newtheorem{thmx}{Theorem}

\newtheorem{corx}[thmx]{Corollary}

\theoremstyle{definition}
\newtheorem {deft}[tw] {Definition}
\newtheorem {rem} [tw]{Remark}

\newcommand{\QG}{\Bbb G}
\newcommand{\hQG}{\widehat{\QG}}
\newcommand{\GG}{\QG}
\newcommand{\whG}{\hQG}

\DeclareMathOperator{\B}{B}
\DeclareMathOperator{\M}{M}
\DeclareMathOperator{\K}{K}

\newcommand{\cst}{\ifmmode\mathrm{C}^*\else{$\mathrm{C}^*$}\fi}

\newcommand{\LL}{\operatorname{L}}
\newcommand{\I}{\mathds{1}}
\newcommand{\ov}{\overline}

\newcommand{\vp}{\varphi}
\newcommand{\oon}{\operatorname}

\newcommand{\bc} {\Bbb C}
\newcommand{\bn}{\Bbb N}
\newcommand{\br}{\Bbb R}

\newcommand{\RR}{\mathbb{R}}
\newcommand{\CC}{\mathbb{C}}
\newcommand{\NN}{\mathbb{N}}

\newcommand{\mrm}{\mathrm}
\newcommand{\lec}{\preceq}
\newcommand{\eps}{\varepsilon}

\numberwithin{equation}{section}

\DeclareMathOperator{\C}{C}

\DeclareMathOperator{\CB}{CB}
\DeclareMathOperator{\A}{A}

\newcommand {\Tr} {{\textrm{Tr}}}

\newcommand {\id} {{\mathrm{id}}}
\newcommand {\Rep} {{\textup{Rep}}}
\newcommand {\Ker} {{\textup{Ker}}}
\newcommand {\Pol} {{\textup{Pol}}}
\newcommand {\Irr} {{\textup{Irr}}}

\newcommand{\mc}{\mathcal}
\newcommand{\msf}{\mathsf}
\newcommand{\mf}{\mathfrak}

\newcommand{\Hil}{\mathsf{H}}
\newcommand{\Ind}{\mathcal{I}}

\newcommand{\Com}{\Delta}

\newcommand{\ww}{\mathrm{W}}

\newcommand{\Ww}{\mathds{W}}
\newcommand{\wW}{\text{\reflectbox{$\Ww$}}\:\!}

\newcommand{\la}{\langle}
\newcommand{\ra}{\rangle}

\DeclareMathOperator{\N}{N}

\newcommand{\ot}{\otimes}

\newcommand{\wot}{\mathbin{\bar{\otimes}}}

\newcommand{\wt}{\widetilde}

\numberwithin{equation}{section}

\begin{document}

%\usetikzlibrary{arrows,positioning}

\author{Jacek Krajczok}
\address{
School of Mathematics and Statistics, University of Glasgow, University Place, Glasgow G12 8QQ, United Kingdom}
\curraddr{Vrije Universiteit Brussel, Pleinlaan 2, 1050 Brussels, Belgium}
\email{jacek.krajczok@vub.be}

\author{Adam Skalski}
\address{Institute of Mathematics of the Polish Academy of Sciences, ul. \'Sniadeckich 8, 00-656 Warszawa, Poland}
\email{a.skalski@impan.pl}

\title{\bf Separation properties for positive-definite functions on locally compact quantum groups and for associated von Neumann algebras}

\begin{abstract}
Using the Godement mean on the  Fourier-Stieltjes algebra	of a locally compact quantum group we obtain strong separation results for quantum positive-definite functions associated to a subclass of representations, strengthening for example the known relationship between amenability of a discrete quantum group and existence of a net of finitely supported quantum positive-definite functions converging pointwise to $\I$. We apply these results to show that von Neumann algebras of unimodular discrete quantum groups enjoy a strong form of non-$w^*$-CPAP, which we call the matrix $\eps$-separation property.
\end{abstract}


\maketitle

\section{Introduction}

The connection between properties of positive-definite functions on a locally compact group $G$, geometric  properties of $G$ and approximation properties of operator algebras associated with $G$ are well-known and form one of the key aspects of analytic geometric group theory.  The situation is most satisfactory for discrete groups, where for example injectivity (in other words, $w^*$-CPAP) of the group von Neumann algebra is equivalent to the existence of a net of finitely supported normalised positive-definite functions converging pointwise to $\I$, and further to the amenability of the group in question (see for example \cite{bo}). A similar correspondence holds for the Haagerup property, with the finitely supported positive-definite functions replaced by these which vanish at infinity (\cite{book}, \cite{choda}).

Analogues of these statements remain true for locally compact quantum groups in the sense of \cite{KustermansVaes}, with the most satisfactory equivalences available for (unimodular) discrete quantum groups (see \cite{Brannan} and references therein). Once again a central role is played by nets of `positive-definite functions' on a locally compact quantum group $\GG$ which have good decay properties and in the limit approximate the constant function $\I$. The corresponding operator algebraic picture  concerns studying normal unital completely positive (UCP) maps on a von Neumann algebra which are `small' in a certain sense, and yet in the limit approximate the identity operator.

The main question studied in this paper is a possibility of weakening the limit property above. Instead of approximating the identity operator (or a constant function $\I$) we want to ask what happens if we can only achieve being `uniformly separated from $0$ in the limit'. Motivated by the classical results of \cite{Derighetti} (see also \cite{Bozejko}), which used the Godement mean of \cite{Godement} to characterise properties of some subclasses of positive-definite functions, we obtain the first of the main results of the paper, which shows that the existence of a net of (`compactly supported') quantum positive-definite functions on $\GG$ which in the limit is `$\eps$-away from $0$' already implies the amenability of $\GG$.  We also obtain a corresponding result for the Haagerup property of \cite{QHap}. For simplicity we formulate these below only for discrete quantum groups.

\begin{thmx}\label{thmA}
	Let $\GG$ be a discrete quantum group. If $\GG$ is not amenable (respectively, not Haagerup), then there is no $\eps>0$ and no net $(f_i)_{i\in I}$ of normalised finitely supported (respectively, vanshing at infinity) positive definite functions on $\GG$ such that
\[
\forall_{\alpha \in \Irr(\whG)} \exists_{i_0\in I}\forall_{i\ge i_0}\quad
 p_\alpha f_i  \ge \eps p_\alpha.
\]
\end{thmx}

In the next step we turn to the analogous operator algebraic question. Recall that a von Neumann algebra $\M$ is said to have the \emph{$w^*$-CPAP}  (equivalently, by \cite{Connes}, is injective) if there exists a net $(\Phi_\lambda)_{\lambda\in \Lambda}$ of normal, finite rank, UCP maps on $\M$ which approximate  the identity map in the pointwise-ultraweak topology.

We say that a von Neumann algebra $\M$ has the \emph{matrix $\eps$-separation property} (for a fixed  $\eps \in (0,1)$)  if for every net $(\Phi_i)_{i \in I}$ of normal, finite rank, UCP maps on $\M$ there is a Hilbert space $\Hil$ and $x\in \M\bar{\otimes }\B(\Hil),\omega\in (\M\bar{\otimes} \B(\Hil))_*$ with $\|x\|=\|\omega\|=1$ such that $\limsup_{i \in I} |\la x-(\Phi_\lambda \ot \id)(x),\omega\ra|> \eps$. It is thus easy to see that if $\M$ has the matrix $\eps$-separation property it cannot be injective. The converse statement -- for a fixed $\eps$ -- appears however non-obvious (an easy `diagonal' argument shows that every non-injective von Neumann algebra must have the matrix $\eps$-separation property for \emph{some} $\eps \in (0,1)$). Appealing to the \emph{Haagerup trick}, which allows us to average arbitrary UCP maps on a von Neumann algebra into quantum Herz-Schur multipliers, and exploiting the separation results for quantum positive-definite functions mentioned above, we are led to the following theorem.

\begin{thmx}\label{thmB}
Let $\GG$ be a compact quantum group of Kac type, $\eps \in (0,1)$. Then $\LL^\infty(\GG)$ has the matrix $\eps$-separation property if and only if $\LL^\infty(\GG)$ is non-injective.
\end{thmx}

The above theorem also has a Haagerup property counterpart. It can be strengthened for these compact quantum groups which admit a uniform bound on the dimension of irreducible representations; in particular we obtain the following result (dropping the word `matrix' from the $\eps$-separation property amounts to setting $\Hil = \bc$ in the definition above).

\begin{corx}\label{corC}
Let $\Gamma$ be a discrete group, $\eps \in (0,1)$. Then $\operatorname{vN}(\Gamma)$ has the $\eps$-separation property if and only if $\operatorname{vN}(\Gamma)$ is non-injective.	
\end{corx}

This naturally leads to a question whether the same equivalence persists for general von Neumann algebras; we formulate it in the end of the paper, providing also certain equivalent reformulations.

The specific plan of the paper is as follows: in the second section we first recall certain preliminary facts concerning locally compact quantum groups and establish a simple technical lemma, and then we pass to studying generalized quantum Fourier-Stieltjes algebras in the spirit of \cite{Eymard} or \cite{KaniuthLau}. Also here the Godement mean makes an appearance and we show in Proposition \ref{prop:Godement} how its behaviour relates to properties of the locally compact quantum group in question. Section \ref{sec:separationpdf} is devoted to studying separation properties for quantum positive-definite functions associated to specific classes of representations. We first establish a general approximation result for such functions in Proposition \ref{prop:approxpdf}, and then use it to prove the main result of this Section, providing a sufficient condition for non-vanishing of the Godement mean on a specific generalized Fourier-Stieltjes algebra, Theorem \ref{thm:Godementnv}. Together with the facts shown in Section \ref{sec:prelim} this implies several corollaries, in particular Theorem \ref{thmA} above follows from Corollaries \ref{cor1}, \ref{cor2} and \ref{cor3}. Finally in Section \ref{sec:sepOA} we study separation properties for von Neumann algebras of discrete quantum groups. We begin by looking at  separation conditions expressed in terms of quantum Herz-Schur multipliers (Proposition \ref{prop5}), use it to motivate the definition of (matrix) $\eps$-separation property (Definition \ref{def1}) and discuss the easy consequences. We then consider specifically unimodular discrete quantum groups, first assuming that we have a bound on the size of irreducibles  in Theorem \ref{thm2} and then dropping this assumption in Theorems \ref{thm:matrix} and \ref{thm:Haagerup}. In particular here we prove Theorem \ref{thmB} (and Corollary \ref{corC} follows from Theorem \ref{thm2}).

\section{Preliminaries and quantum Fourier-Stieltjes algebras} \label{sec:prelim}

We will work in the setting of locally compact quantum groups, introduced by Kustermans and Vaes \cite{KustermansVaes} (see also \cite{VanDaele}). By defnition, any locally compact quantum group $\QG$ comes together with a von Neumann algebra $\LL^{\infty}(\QG)$, \emph{comultiplication} $\Delta\colon \LL^{\infty}(\QG)\rightarrow \LL^{\infty}(\QG) \wot \LL^{\infty}(\QG) $ which is a normal, unital $*$-homomorphism and two n.s.f.\ (i.e.\ normal semifinite faithful) weights $\vp,\psi$ which are called \emph{Haar integrals} and satisfy left (resp.~right) invariance conditions. The predual of $\LL^{\infty}(\QG)$ is denoted by $\LL^1(\QG)$ and the GNS Hilbert space of $\vp$ is $\LL^2(\QG)$ -- it can be also identified with the GNS Hilbert space of $\psi$. The corresponding GNS map is denoted $\Lambda_\vp\colon \mf{N}_\vp\rightarrow \LL^2(\QG)$, where $\mf{N}_\vp:= \{x \in \LL^\infty (\QG)\,|\, \vp(x^*x) < \infty \}$. With any locally compact quantum group one can associate its dual $\whG$ and by construction $\LL^2(\whG)$ is equal to $\LL^2(\QG)$. The assignment $\QG\mapsto \whG$ extends the classical Pontryagin duality of locally compact abelian groups and is itself a duality in the sense that the dual of $\whG$ is canonically isomorphic with $\QG$. We will follow the convention which favours left objects over the right ones. 

An important result in the theory states the existence of \emph{Kac-Takesaki operator} $\ww\in \LL^{\infty}(\QG)\wot \LL^{\infty}(\whG)$. It is a unitary operator which implements comultiplication via $\Delta(x)=\ww^* (\I\otimes x)\ww$ for $x\in \LL^{\infty}(\QG)$. One can construct a weak$^*$-dense $\mrm{C}^*$-subalgebra of $\LL^{\infty}(\QG)$ via $\mrm{C}_0(\QG)=\ov{\{(\id\otimes\omega)\ww\,|\,\omega\in \LL^1(\whG)\}}$. Then $\ww\in \M(\mrm{C}_0(\QG)\otimes \mrm{C}_0(\whG))$ and comultiplication restricts to a non-degenerate $*$-homomorphism $\mrm{C}_0(\QG)\rightarrow \M(\mrm{C}_0(\QG)\otimes \mrm{C}_0(\QG))$. We write $\C_b(\QG)$ for $\M(\C_0(\QG))$. The $\mrm{C}^*$-algebra introduced above has a universal counterpart, $\mrm{C}_0^u(\QG)$ (see \cite{Johan}). It is equipped with its own comultiplication and the \emph{reducing map} $\lambda_{\QG}\colon \mrm{C}_0^u(\QG)\rightarrow\mrm{C}_0(\QG)$, a surjective $*$-homomorphism commuting with respective comultiplications. Kac-Takesaki operator admits several lifts, in particular the right-universal version $\wW\in \M(\mrm{C}_0(\QG)\otimes \mrm{C}_0^u(\whG))$ with similar properties to $\ww$. Both operators are related via the formula  $(\id\otimes \lambda_{\whG})\wW=\ww$. One says that $\QG$ is \emph{compact} if $\mrm{C}_0(\QG)$ is unital (equivalently, the Haar integrals are states -- and thus coincide); we then simply write $\mrm{C}(\QG)$ instead of $\mrm{C}_0(\QG)$.  Further one says that $\QG$ is \emph{discrete} if $\whG$ is compact. In this case $\whG$ is said to be \emph{Kac} if its Haar integral is tracial; equivalently, $\QG$ is \emph{unimodular}, i.e.\ its left and right Haar integrals coincide. 

A \emph{(unitary) representation} of $\QG$ on a Hilbert space $\Hil$ is a unitary element $U\in \M(\mrm{C}_0(\QG)\otimes \mc{K}(\Hil))$ which satisfies $(\Delta\otimes \id)(U)=U_{13} U_{23}$, where we use the standard leg-numbering notation. \emph{Coefficients} of $U$ are then elements of the form $(\id\otimes \omega_{\xi,\eta})(U^*)\in \mrm{C}_b(\QG)\,(\xi,\eta\in \Hil)$ -- we adopt this definition because of our convention concerning the Fourier algebra. There is a one-to-one correspondence between representations $U$ and non-degenerate $*$-representations $\phi_U\colon \mrm{C}_0^u(\whG)\rightarrow \B(\Hil)$, given by the formula $U=(\id\otimes \phi_U)(\wW)$. We shall write $\Rep(\QG)$ for the family of unitary equivalence classes of unitary representations of a locally compact quantum group $\QG$. We will follow the conventional abuse of notation and identify representation with its unitary equivalence class.
Given two representations $U,V$ we say that $U$ is weakly contained in $V$ (written $U\lec V$) if $\phi_U$ is weakly contained in $\phi_V$ (written as $\phi_U\lec \phi_V$), i.e.\ when $\ker(\phi_V)\subset \ker (\phi_U)$ (see also \cite[Theorem 1.2]{Fell} for other equivalent conditions). We will always assume that $\Rep(\QG)$ is equipped with the Fell topology. Note that if $\GG$ is discrete we have $\mrm{c}_0(\GG)= \bigoplus_{\alpha \in \Irr(\whG)} \M_{\dim(\alpha)}$, where $\Irr(\whG)$ denotes the set of equivalence classes of \emph{irreducible} unitary representations of $\hQG$. On the other hand if $\QG$ is compact, the coefficients of all irreducible unitary representations of $\QG$ span a canonical dense Hopf *-subalgebra of $\mrm{C}(\QG)$, denoted $\Pol(\QG)$.

We will also work with another subclass of locally compact quantum groups, called \emph{algebraic quantum groups}. These are defined via a multiplier Hopf$^*$-algebra $(\mathfrak{C}^{\infty}_c(\QG),\Delta)$ and a left Haar integral which satisfy certain conditions, see \cite{KVD}. Every algebraic quantum group gives rise to a locally compact quantum group in such a way that $\mathfrak{C}^{\infty}_c(\QG)$ is dense in $\mrm{C}_0(\QG)$ and the respective comultiplications agree. In particular, compact and discrete quantum groups are special cases of algebraic quantum groups -- the corresponding multiplier Hopf$^*$-algebras are given by respectively $\Pol(\QG)$ and $\mrm{c}_{c}(\QG)=alg-\bigoplus_{\alpha \in \Irr(\whG)} \M_{\dim(\alpha)}$.

We end the preliminary part with a general technical fact, which is easy and likely well-known. As we could not find an exact reference,  we provide a proof.

\begin{prop}\label{abstract}
	Suppose $X$ is a dual Banach space and that $Y$ is a closed subspace of $X$ whose unit ball is weak$^*$-dense in the unit ball of $X$. Then we can identify $X_*$ isometrically with a weak$^*$-dense subspace of $Y^*$; moreover the unit ball of $X_*$ is weak $^*$-dense in the unit ball of $Y^*$. If $Y$ is a $C^*$-algebra which is weak$^*$-dense in a von Neumann algebra $X$, then the positive part of the unit ball of $X_*$ is weak $^*$-dense in the positive part of the unit ball of $Y^*$.
\end{prop}

\begin{proof}
	Indeed, given $\phi \in X_*$ we can always restrict it to $Y$; this procedure is injective, as $Y$ is weak$^*$-dense, and isometric due to the assumption of the weak$^*$-density of the unit ball of $Y$ in the unit ball of $X$. %It is in fact isometric.% Indeed, if $x\in X$ and $(y_i)_{i \in I}$ is a net of elements of $Y$ weak $^*$-convergent to $x$ then $\|x\|\leq \limsup{i \in I} \|y_i\|$ and $|\phi(x)\| = \lim_{i \in I}|\phi(y_i)| \leq \limsup \|\phi|_Y\| \|y_i\| \leq \|\phi|_Y\|
	On the other hand, given a contractive $\omega \in Y^*$ we can first extend it by Hahn-Banach to a contractive $\tilde{\omega} \in X^*$ and then approximate $\tilde{\omega}$ in weak$^*$-topology by contractive functionals in $X_*$ using Goldstine's theorem; this does the job. 
	
	For the last part note first that due to the Kaplansky Density Theorem the density of the balls holds automatically. Moreover we can work with states and then the only non-trivial element is the fact that we can approximate states on $X$ by states in $X_*$. Let then $\phi \in X^*$ be a state and let $(\phi_i)_{i \in \Ind}$ be a net of contractive non-zero elements in $X_*$ convergent to $\phi$ in weak$^*$-topology. Note that we must have $\lim_{i \in \Ind}\|\phi_i\| = 1 = \|\phi\|$ (as $1 \geq \|\phi_i\| \geq |\phi_i(\I)|\xrightarrow[i \in \Ind]{} \phi(\I) = 1$). By \cite[Proposition 4.11]{Takesaki} the absolute value of $\phi_i$ -- which is a contractive non-zero positive element of $X_*$ -- converges in the weak$^*$-topology of $X^*$ to the absolute value of $\phi$ (i.e.\ to $\phi$). Then $(|\phi_i|(\I)^{-1}|\phi_i|)_{i \in \Ind}$ is the desired net.   
\end{proof}

\subsection*{Quantum Fourier-Stieltjes algebras}

Let us start with the following definition, in the spirit of \cite[D\'efinition 2.2]{Eymard}.

\begin{deft}
	Suppose that $\emptyset\neq S \subset \Rep(\QG)$. Denote by $B_S(\QG)$ the set of all coefficients of unitary representations weakly contained in  $S$ (equivalently coefficients of representations in $\overline{S})$:
	\[ B_S(\QG)= \{(\id \ot \omega_{\xi, \eta})(U^*)\,|\,U\preceq S, \xi, \eta \in \Hil_U\} \subset \C_b(\QG).\]
	On the other hand define $J_S= \bigcap_{U \preceq S} \Ker (\phi_U)$ and let $\C^S_0(\hQG)= \C_0^u(\hQG)/J_S$. 	
\end{deft}

Note that in particular \emph{the Fourier-Stieltjes algebra of} $\QG$, going back at least to \cite{DawsMultipliers0}, is $B(\QG)= B_{\{\wW\}}(\QG) = \{(\id \ot \omega_{\xi, \eta})(U^*)\,|\,U\in \Rep(\QG), \xi, \eta \in \Hil_U\}$. We shall denote by $\lambda$ the left regular representation of $\QG$ and write $B_\lambda(\QG)$ for 
$B_{\{\ww\}}(\QG)$. We will also need to consider $S_{mix}$, the collection of all \emph{mixing} representations of $\QG$, i.e.\ these whose all coefficients belong to $\C_0(\QG)$ (see \cite[Definition 4.1]{QHap}).  We will write $B_0(\QG)$ for $B_{S_{mix}}(\QG)$. Suitable versions of spaces $B_S(\GG)$ were exploited in \cite{BrannanRuan}, following the earlier classical ideas of \cite{bg}, to construct a continuum of different $C^*$-completions of Hopf $^*$-algebras associated to free orthogonal quantum groups.

It is also easy to see -- directly from the definitions -- that $J_S= \bigcap_{U \in S} \Ker (\phi_U)$ and that $\C^S_0(\hQG)$ can be alternatively described as a completion of $\LL_1^{\sharp}(\QG)$ (say viewed as a $^*$-subalgebra of $ \C_0^u(\hQG)$) with respect to the norm $\|f\|_S= \sup_{U \in S}\|\phi_U(f)\|$. If $\QG$ is discrete we can also describe $\C^S_0(\hQG)$ as the analogous completion of $\Pol(\hQG)$. Using the fact that a direct sum of representations weakly contained in $S$ is also weakly contained in $S$ we deduce that $B_S(\QG)$ is a vector subspace of $\C_b(\QG)$.

Note that given $U\preceq S$ and $\xi, \eta \in\Hil_U$ the functional $\omega_{\xi,\eta}\circ\phi_U\in \mrm{C}_0^u(\whG)^*$ vanishes on $J_S$ and hence defines a new functional on the quotient space $\C_0^S(\hQG)=\mrm{C}_0^u(\whG)/J_S$. We will often abuse the notation and denote it again by $\omega_{\xi,\eta}\circ\phi_U$.

\begin{lem} \label{lem:dual}
		Suppose that $\emptyset\neq S \subset \Rep(\QG)$.
The formula 
\[ (\id \ot \omega_{\xi, \eta})(U^*) \mapsto \omega_{\xi, \eta} \circ \phi_U, \;\;\; U\preceq S, \xi, \eta \in \Hil_U \]
defines a linear bijection between $B_S(\QG)$ and $\C^S_0(\hQG)^*$.	
\end{lem}

\begin{proof}
Suppose that $U,V\preceq S$ and $\xi, \eta \in \Hil_U$, $\xi', \eta' \in \Hil_V$,
\[ (\id \ot \omega_{\xi, \eta})(U^*) = (\id \ot \omega_{\xi', \eta'})(V^*).
\]
Then we have -- understanding the functionals $\omega_{\xi, \eta} \circ \phi_U$ and $\omega_{\xi', \eta'} \circ \phi_V$ as functionals on $\C_0^u(\hQG)$ -- the following equality:
\[ 
(\id \ot \omega_{\xi, \eta}\circ \phi_U)(\wW^*) = (\id \ot \omega_{\xi', \eta'}\circ \phi_V)(\wW^*). 
\]
As left slices of $\wW^*$ are dense in $\C_0^u(\hQG)$ by \cite[Proposition 4.2]{Johan}, we have $\omega_{\xi, \eta}\circ \phi_U =\omega_{\xi', \eta'}\circ \phi_V$. We view the latter as bounded functionals on $\C_0^u(\hQG)$, but as $U,V \preceq S$, they also descend to (equal) functionals in $\C^S_0(\hQG)^*$.
This implies that the map introduced in the lemma is well-defined and injective. 

As it is clearly linear, to verify surjectivity it suffices to consider states on $\C_0^S(\hQG)$. Every such state $\omega$ is of the form $\omega_{\xi, \xi}\circ \pi$, where $\pi\colon\C_0^S(\hQG) \to \B(\Hil)$ is a representation, and $\xi \in \Hil$. But then we can consider $\pi\circ q_S\colon \C_0^u(\hQG) \to \B(\Hil)$ and note that $\pi \circ q_S = \phi_V$ for $V\in \Rep(\QG)$, $V \preceq S$. Naturally we have $\omega = \omega_{\xi, \xi}\circ \phi_V$.
\end{proof}

The result of the last lemma allows us to introduce the norm on $B_S(\QG)$ induced by the norm of $\C^S_0(\hQG)^*$, namely set
\[
\|a\|_{B_S(\QG)}=\|\omega_{\xi,\eta}\circ\phi_U\|_{\C_0^S(\hQG)^*}\quad(a=(\id\otimes \omega_{\xi,\eta})(U^*)\in B_S(\QG)).
\]

It is also worth noting that in the case where $S=\Rep(\QG)$ the inverse of the map above is given simply by the formula 
\begin{equation} \C_0^u(\hQG)^*\ni \mu \mapsto (\id \ot \mu)(\wW^*) \in B(\QG).\label{inversemap}\end{equation}

Recall the definition of the Fourier algebra (for the early definitions in the quantum group context see for example \cite{DawsMultipliers} or \cite{DKSS}, although these papers use a different convention; here we follow rather \cite{Brannan}):
\begin{equation}
 \A(\QG)=  \{(\id \ot \omega_{\xi, \eta})(\ww^*)\,|\, \xi, \eta \in \LL^2(\QG)\} \subset \C_0(\QG).
\end{equation}
It is well-known that the formula appearing in Lemma \ref{lem:dual} identifies $\A(\QG)$ with $\LL^1(\hQG)$; this can be proved following the same lines as above, using the fact that $\LL^\infty(\hQG)$ acts on $\LL^2(\QG)$ in a standard form. Again we view $\A(\QG)$ as a normed space equipped  with the norm of $\LL^1(\hQG)$. Note that this space has been successfully exploited in the quantum setting for example in \cite{Caspers} to study questions related to weak amenability.

\begin{lem}\label{lem:inclusions}
Let $S, T \subset \Rep(\QG)$, $\lambda \in S \subset T$. We have natural isometric inclusions $\A(\QG) \subset B_\lambda(\QG) \subset B_S(\QG) \subset B_T(\QG)\subset B(\QG)$. Moreover for $f \in B(\QG)\subset \C_b(\QG)$ we have $\|f\|_{\C_b(\QG)} \leq \|f\|_{B(\QG)}$. 
\end{lem}
\begin{proof}
The first part is an easy consequence of Lemma \ref{abstract} (applied to $Y=\C_0(\hQG)$ and $X = \LL^\infty(\hQG)$), Lemma \ref{lem:dual} and properties of dual spaces of quotient $\C^*$-algebras.	The second follows from the formula \eqref{inversemap}.
\end{proof}

An element $x$ in $B(\QG)$ is said to be \emph{positive-definite} if it is positive as a functional on $\C_0^u(\hQG)$. For an extended discussion of this notion and equivalent characterizations see \cite{MattPekka}; note in particular that it makes sense to talk of positive-definite functions in $\LL^{\infty}(\QG)$ -- which then automatically turn out to belong to $B(\QG)$, at least if $\QG$ is coamenable.
Note however that the authors of \cite{MattPekka} call the elements as above rather `completely positive-definite functions'.

 We shall say that a (quantum) positive-definite function is \emph{normalised} if the associated functional is a state.
Note that by \cite[Lemma 4.3]{QHap} positive-definite functions in $\C_0(\QG)$ automatically belong to $B_0(\QG)$ (as the relevant GNS representations are mixing).
	
	The next statement is well-known for discrete quantum groups.

\begin{prop}\label{prop:A(G)}
	Let $\QG$ be an algebraic quantum group. Then
	\begin{enumerate}
	\item $\mathfrak{C}^{\infty}_c(\GG)\subset \A(\QG)$;
	\item $\Lambda_{\vp}(\mathfrak{C}^{\infty}_c(\GG))=\Lambda_{\psi}(\mathfrak{C}^{\infty}_c(\GG))=\Lambda_{\widehat{\vp}}(\mathfrak{C}^{\infty}_c(\hQG))=\Lambda_{\widehat{\psi}}(\mathfrak{C}^{\infty}_c(\hQG)) $;
	\item 
 $\A(\QG)$ coincides with the linear span of ($\A(\QG)$-) norm limits of normalised positive definite functions in $\mathfrak{C}^{\infty}_c(\GG)$.
\end{enumerate}
\end{prop}

\begin{proof}
From \cite[Lemma 2.3]{KVD}, we have that for every $a, b \in \mathfrak{C}^{\infty}_c(\GG)$
	\[
(\id \ot \omega_{\Lambda_\vp(b), \Lambda_\vp(a)})(\ww) = (\id \otimes \varphi) (\Com(b^*)(\I\otimes a)),
	\]
so that
\begin{equation}\label{KVDformula}
(\id \ot \omega_{\Lambda_\vp(a), \Lambda_\vp(b)})(\ww^*) = (\id \otimes \varphi) ((\I\otimes a^*)\Com(b)).
\end{equation}	

By the axioms of algebraic quantum groups we have 
	\[\mathfrak{C}^{\infty}_c(\GG) = \textup{span} \{(\id \otimes \varphi) ((\I\otimes a^*)\Com(b))\,|\, a, b \in \mathfrak{C}^{\infty}_c(\GG)\}\]
	hence the first claim follows from equation \eqref{KVDformula}.
	
Recall that one can identify $\mathfrak{C}^{\infty}_c(\whG)$ with the space of functionals $\{a\vp\,|\,a\in \mathfrak{C}^{\infty}_c(\GG)\}$ on $\mathfrak{C}^{\infty}_c(\GG)$. Then $\Lambda_{\widehat\psi}(a\vp)=\Lambda_\vp(a)$ (\cite[Page 1077]{KVD}), which readily implies $\Lambda_{\widehat\psi}(\mathfrak{C}^{\infty}_c(\whG))=\Lambda_\vp(\mathfrak{C}^{\infty}_c(\GG))$. Next, by \cite[Definition 8.13]{KVD} for any $z\in \CC$, the operator $\delta^{z}$ is a multiplier of $\mathfrak{C}^{\infty}_c(\GG)$ and $\Lambda_{\psi}(a)=\Lambda_\vp(a\delta^{1/2})$ by \cite[Page 1134]{KVD} (see also the proof of \cite[Lemma 9.15]{KVD}). This shows that $\Lambda_\vp(\mathfrak{C}^{\infty}_c(\GG))=\Lambda_\psi(\mathfrak{C}^{\infty}_c(\GG))$. Since the dual of an algebraic quantum group is also an algebraic quantum group, this finishes the claim.
	
It follows from the very construction that the image of $\Lambda_\vp|_{\mathfrak{C}^{\infty}_c(\QG)}$ is dense in $\LL^2(\QG)$. Thus we can approximate in norm any vector state $\omega_\xi$ on $\LL^\infty(\hQG)$ for $\xi \in \LL^2(\hQG)$ by vector states associated  with vectors in $\Lambda_\vp(\mathfrak{C}^{\infty}_c(\GG))$.  Together with the previous paragraph it shows that each normalised positive-definite function in $\A(\QG)$ is a  norm limit of normalised positive-definite functions in $\mathfrak{C}^{\infty}_c(\GG)$, which together with the polarisation identity  proves the third part of the lemma.

\end{proof}

Note that the Fourier-Stieltjes algebra $B(\QG)$ is closed with respect to natural left and right actions of $\LL^1(\QG)$ induced from $\LL^{\infty}(\QG)$ (see \cite[Proof of Proposition 2.15]{DSV}). The existence of the \emph{Godement mean}  (\cite{Godement}) for quantum $B(\QG)$, i.e.\ a specific bi-invariant linear contractive unital functional $M\colon B(\QG)\to \bc$ of norm $1$,  was established in \cite[Proposition 2.15]{DSV} (see also \cite{DasDaws}). Note the following consequence of the proof of \cite[Proposition 2.15]{DSV}.

\begin{prop}\label{Godementmean}
	Let $\QG$ be a locally compact quantum group and let $U$ be a unitary representation of $\QG$ on a Hilbert space $\Hil$. Then the following are equivalent:
	\begin{enumerate}
		\item the Godement mean vanishes on all matrix coefficients of $U$;
		\item $U$ admits no invariant vectors.
	\end{enumerate}
\end{prop}
\begin{proof}
	Indeed, the proof of \cite[Proposition 2.15]{DSV} shows that for all $\xi, \eta \in \Hil$ we have
	\[ M((\id \otimes \omega_{\xi,\eta})(U^*)) = \langle \xi| p \eta \rangle, \]
	where $p\in \B(\Hil)$ is the projection onto the invariant vectors of $U$.
\end{proof}

The following is  an immediate consequence of Proposition \ref{Godementmean} and the transitivity of the weak containment relation.

\begin{prop}\label{prop2}
Let $S$ be a non-empty collection of representations of $\QG$. Then the following are equivalent:
\begin{enumerate}
	\item the Godement mean vanishes on $B_S(\QG)$;
	\item $S$ does not weakly contain the trivial representation.
\end{enumerate}
\end{prop}
\begin{proof}
	Indeed, if (1) does not hold then by the previous proposition there is a representation of $\QG$ which is weakly contained in $S$ and contains the trivial representation. Thus $S$ weakly contains the trivial representation.
	
On the other hand if $S$ weakly contains the trivial representation, then $\I\in B_S(\QG)$ and we have $M(\I)=1$.
	
\end{proof}

%It seems -- and has been confirmed by Matt Daws -- that these spaces have not yet appeared in the literature, so it would be nice to prove certain basic facts about $B_S(\QG)$: for example, to show that these are closed subspaces of $B(\QG)$ (where the latter comes from the universal representation) and they are isometrically isomorphic to duals of the corresponding quotients of $C^*_u(\QG)$ (otherwise known as $C^u_0(\widehat{\QG})$). These should be easy facts, I just need to have time to check them.

The last proposition allows us to characterise for a given $\QG$ the coamenability of the dual, and the Haagerup property via the properties of the Godement mean. See \cite[Proposition 2]{Derighetti} for classical analogues and recall that we write $B_\lambda(\QG)$ for $B_{\{\ww\}}(\QG)$ and $B_0(\QG)$ for $B_{S_{mix}}(\QG)$. 

\begin{prop}\label{prop:Godement}
Let $\QG$ be a locally compact quantum group. Then 

\begin{enumerate}
\item	$\widehat{\QG}$ is not coamenable if and only if the Godement mean vanishes on $B_\lambda(\QG)$;
\item $\QG$ does not have the Haagerup property if and only if the Godement mean vanishes on $B_0(\QG)$.
\end{enumerate}
\end{prop}

\begin{proof}
	The first statement follows from Proposition \ref{prop2} and  \cite[Theorem 3.1]{BT}. The second follows from Proposition \ref{prop2} and the very definition of the Haagerup property for quantum groups, \cite[Definition 6.1]{QHap}. 
\end{proof}

\section{Separation properties for positive-definite functions on locally compact quantum groups} \label{sec:separationpdf}

This section is devoted to establishing various separation properties for quantum positive-definite functions associated with particular classes of representations.

We begin with providing several equivalent descriptions of (normalised) positive-definite functions in $B_S(\QG)$ under a mild assumption on the set $S$.
Note that the left regular representation $\ww$ is (up to equivalence) equal to its multiple, hence the  proposition below applies to the set of positive definite functions in $B_\lambda(\GG)$ with norm bounded by $1$.

\begin{prop}\label{prop:approxpdf}
Let $\emptyset\neq S\subset \Rep(\QG)$ be a set closed under finite direct sums.
% Assume that the set of elements of $\C_b(\GG)$ of the form $(\id \otimes \omega_{\xi,\eta})(U)$, where $U \in S$, $\xi, \eta \in \Hil_U$, is a linear space (as is the case if $S$ is closed under direct sums or when $S=\{\ww\}$).
The following convex sets are equal:
\begin{enumerate}

\item[(1)] the set of positive definite functions $a$ in $B_S(\QG)$ with $\|a\|_{B_S(\GG)}\le 1$;
\item[(2)] $\{(\id\otimes \omega_{\xi} )(U^*)\, | \, U\lec S, \xi\in \Hil_U, \|\xi\|\le 1\}$;
\item[(3)] $\{ a\in B_S(\GG)\,|\, a\,\textnormal{ is a weak}^*\textnormal{ limit of a net of positive definite functions}$

\hspace{3cm}$\textnormal{of the form } (\id\otimes \omega_\xi) (U^*)\textnormal{ where } \, U\in S, \xi\in \Hil_U, \|\xi\|\le 1\}$;
\item[(4)] $\{ a\in B_S(\GG)\,|\, a\,\textnormal{ is a weak}^*\textnormal{ limit of a net of positive definite functions}$

\hspace{3cm}$\textnormal{of the form } (\id\otimes \omega) (U^*)\textnormal{ where } \, U\lec S, \omega\in \B(\Hil_U)_*^+,\|\omega\|\le 1\}$;

\item[(5)] $\{a\in \mrm{C}_b(\GG)\,|\,
a \textnormal{ is a strict limit of a net of positive definite functions}$

\hspace{3cm}$\textnormal{of the form } (\id\otimes \omega_\xi) (U^*)\textnormal{ where } \, U\in S, \xi\in \Hil_U, \|\xi\|\le 1\}$; 
\item[(6)] $\{a\in \mrm{C}_b(\GG)\,|\,
a \textnormal{ is a strict limit of a net of positive definite functions}$

\hspace{3cm}$\textnormal{of the form } (\id\otimes \omega) (U^*)\textnormal{ where } \, U\lec S, \omega\in \B(\msf{H}_U)_*^+, \|\omega\|\le 1\}$. 
\end{enumerate}
\end{prop}

\begin{proof}
Clearly $(1)$ is convex, hence it is enough to argue the equality of sets $(1)$ -- $(6)$. Note first that $(2)$ is contained in $(1)$, $(3)$ in $(4)$ and $(5)$ in $(6)$.

We will now prove that the sets in $(1)$ and $(2)$ are equal. Let $a\in B_S(\QG)$ be a positive definite function. This means that there is $U\lec S,\xi,\eta\in \Hil_U$ and $\omega\in \C^u_0(\whG)^*_+$ such that $a=(\id\otimes \omega)(\wW^*)=(\id\otimes \omega_{\xi,\eta})(U^*)$ and $\|\omega\|\le 1$. Since $(\id\otimes \omega_{\xi,\eta})(U^*)=(\id\otimes \omega_{\xi,\eta}\circ \phi_U)(\wW^*)$, and left slices of $\wW^*$ are dense in $\C^u_0(\whG)$, we have $\omega_{\xi,\eta}\circ\phi_U = \omega$. Let $(\pi_\omega,\zeta_\omega,\Hil_\omega)$ be the GNS representation for $\omega$, in particular $\|\zeta_\omega\|\le 1$ and $\omega=\omega_{\zeta_\omega}\circ\pi_\omega$ holds on $\mrm{C}_0^u(\whG)$. The representation $\pi_\omega$ is weakly contained in $S$ by \cite[Proposition 3.4.9]{DixmierC}. This shows that the sets in $(1)$ and $(2)$ are equal.

Next we show that any member $(\id\otimes \omega_\xi)(U^*)$ of $(2)$ belongs to $(3)$. We assume that $\phi_U\lec S$, where $\phi_U\colon\mrm{C}_0^u(\whG)\rightarrow \B(\Hil_U)$ is the representation associated to $U$. By \cite[Theorem 1.2]{Fell} we can find a collection of vectors $\xi_{i,j}\in \Hil_{U_{i,j}}$ with $U_{i,j}\in S\,(i\in I, j\in \{1,\dotsc, N_i\})$ such that
\begin{equation}\label{eq4}
\bigl\| \sum_{j=1}^{N_i} \omega_{\xi_{i,j}}\circ \phi_{U_{i,j}}\bigr\|_{\mrm{C}_0^u(\hQG)^*}\le \|\omega_{\xi}\circ\phi_U\|_{\mrm{C}_0^u(\hQG)^*},\quad
\omega_\xi\circ\phi_U=w^*-\lim_{i\in I} \sum_{j=1}^{N_i} \omega_{\xi_{i,j}}\circ \phi_{U_{i,j}}.
\end{equation}
After rescaling, we can assume equality in \eqref{eq4} for each $i \in I$.
The net
\begin{equation}
\bigl(\sum_{j=1}^{N_i} (\id\otimes \omega_{\xi_{i,j}})(U_{i,j}^*)
\bigr)_{i\in I}
\end{equation}
is then a bounded net of normalised positive definite functions in $B_S(\GG)$ which by \eqref{eq4} converges to $(\id\otimes \omega_\xi)(U^*)$ in $(B_S(\GG),w^*)$. Since $S$ is closed under finite direct sums, we can consider $U_i=\bigoplus_{j=1}^{N_i}U_{i,j}\in S$ and $\zeta_i=\bigoplus_{j=1}^{N_i}\xi_{i,j}\in  \Hil_{U_{i}}$, so that  $\sum_{j=1}^{N_i} \omega_{\xi_{i,j}} \circ \phi_{U_{i,j}}=\omega_{\zeta_i}\circ \phi_{U_i}$. Furthermore for each $i \in I$ we have
\[
1\ge \|\xi\|_{\Hil_U}=\|\omega_\xi\circ \phi_U\|_{\mrm{C}_0^u(\whG)^*}= 
\bigl\|\sum_{j=1}^{N_i} \omega_{\xi_{i,j}}\circ \phi_{U_{i,j}}
\bigr\|_{\mrm{C}_0^u(\whG)^*}= 
\sum_{j=1}^{N_i}\|\xi_{i,j}\|^2= 
\| \zeta_i\|^2,
\]
which ends the claim.

To see that $(3)$ is contained in $(5)$, take $a\in B_S(\GG)$ which is a weak$^*$ limit of a normalised net $\bigl((\id\otimes\omega_{\xi_i})(U_i^*)\bigr)_{i\in I}$ in $B_S(\GG)$, where $U_i\in S, \xi_i\in \Hil_{U_i}$, $\|\xi_i\|\le 1$.  That means that we can find $U\lec S, \xi,\eta\in \Hil_U$ so that $a=(\id\otimes\omega_{\xi,\eta})(U^*)$ and $\omega_{\xi_i}\circ\phi_{U_i}\xrightarrow[i\in I]{} \omega_{\xi,\eta}\circ\phi_U$ in the weak$^*$ topology of $\mrm{C}_0^u(\whG)^*$. By \cite[Theorem 4.6]{AmiVolker} we deduce that $\bigl( (\id\otimes \omega_{\xi_i})(U_i^*)\bigr)_{i\in I}$ converges strictly to $a$. An analogous argument shows that $(4)$ is contained in $(6)$.

It remains to show that the set in $(6)$ is contained in $(1)$. Take $a\in \mrm{C}_b(\GG)$ which is a strict limit of a net $\bigl((\id\otimes\omega_i)(U_i^*)\bigr)_{i\in I}$, where $U_i\lec S$ and $\|\omega_i\|\le 1$. We want to show that $a$ is a positive definite function in $ B_S(\QG)$ with $\|a\|_{B_S(\QG)}\le 1$. Consider the net of contractive functionals $(\omega_i\circ \phi_{U_i})_{i\in I}$ on $\mrm{C}_0^u(\whG)$. After passing to a subnet we can assume that there is $\omega\in \mrm{C}_0^u(\whG)^*$ such that $\omega_i\circ\phi_{U_i}\xrightarrow[i\in I]{}\omega$ weak$^*$. Clearly $\omega$ is positive and $\|\omega\|\le 1$. Let $(\pi_\omega, \zeta_\omega,\Hil_\omega)$ be the GNS representation for $\omega$. We claim that $\pi_\omega\lec S$ and $a=(\id\otimes \omega_{\zeta_\omega}\circ\pi_\omega )(\wW^*)$, which will end the proof. The first property follows again by \cite[Proposition 3.4.9]{DixmierC}. Denote $b=(\id\otimes \omega_{\zeta_\omega}\circ\pi_\omega )(\wW^*)\in \mrm{C}_b(\GG)$ and take $\rho\in \LL^1(\GG)\subset \mrm{C}_0(\GG)^*, c\in \mrm{C}_0(\GG)$. We have
\[\begin{split}
\la \rho c , a\ra &=\la \rho , ca \ra = 
\lim_{i\in I} \la \rho, c (\id\otimes \omega_{i})(U_i^*)\ra =
\lim_{i\in I}\la \omega_i\circ\phi_{U_i},  (\rho c \otimes \id)(\wW^*) \ra \\
&=
\la \omega_{\zeta_\omega}\circ \pi_\omega , (\rho c\otimes \id)(\wW^*)\ra=
\la \rho c , b\ra .
\end{split}\]

Note that above we used the fact that $\rho c \in \LL^1(\GG)$ (as the left multiplication by $c$ is a normal map on $\LL^{\infty}(\GG)$) and the set of functionals of the form $\rho c$ is norm dense in $\LL^1(\GG)$. In fact, by Cohen-Hewitt factorization theorem, any functional in $\LL^1(\GG)$ can be written as $\rho c$ for some $\rho\in \LL^1(\GG),c\in \mrm{C}_0(\GG)$, see \cite[Proposition 2.1]{HuNeufangRuan}. 
\end{proof}

\begin{cor}\label{cor:limits}
Let $\GG$ be a locally compact quantum group and $V\subset \LL^2(\GG)$ a dense subspace. The following convex sets are equal:
\begin{enumerate}

\item the set of positive definite functions $a$ in $B_\lambda(\QG)$ with $\|a\|_{B_\lambda(\GG)}\le 1$;
\item $\{ a\in B_\lambda(\GG)\,|\, a\,\textnormal{ is a weak}^*\textnormal{ limit of a net of positive definite functions}$

\hspace{3cm}$\textnormal{of the form } (\id\otimes \omega_{\xi}) (\ww^*) \textnormal{ where } \, \xi\in V, \|\xi\|\le 1\}$;
\item $\{a\in \mrm{C}_b(\GG)\,|\,
a \textnormal{ is a strict limit of a net of positive definite functions}$

\hspace{3cm}$\textnormal{of the form } (\id\otimes \omega_{\xi}) (\ww^*)\textnormal{ where } \, \xi\in V, \|\xi\|\le 1\}$.

\end{enumerate}
\end{cor}

\begin{proof}
The above result follows directly from Proposition \ref{prop:approxpdf} for $S=\{\ww\}$ and density of $V$ in $\LL^2(\GG)$.
\end{proof}

The next result should be compared to Proposition \ref{prop:A(G)}.

\begin{prop}\label{prop1}
	Let $\QG$ be an algebraic quantum group. 	Then $B_\lambda(\QG)\subset \mrm{C}_b(\QG)$ coincides with the linear span of strict limits of  normalised  positive definite functions in $\mathfrak{C}^{\infty}_c(\GG)$.
\end{prop}

\begin{proof}
It suffices to show that positive-definite functions in $B_\lambda(\QG)$ of norm not greater than $1$ are precisely strict limits of  positive definite functions in $\mathfrak{C}^{\infty}_c(\GG)$ of norm not greater then $1$. This follows from Corollary \ref{cor:limits} applied to $V=\Lambda_\vp(\mathfrak{C}^{\infty}_c(\GG))$ and from \cite[Lemma 2.3]{KVD}, arguing again as in the proof of Proposition \ref{prop:A(G)}.
\end{proof}

\begin{rem}
	It is worth noting that	in spite of Proposition \ref{prop1} the space $B_\lambda(\QG)$ need not be closed under strict limits; indeed, this is not the case whenever $\QG$ is a non-amenable discrete group (it suffices to consider any net of finitely supported functions converging pointwise to $\I$).
\end{rem}

The next theorem is our main result in this section. We show that, roughly speaking, if $\whG$ is not coamenable, one cannot find a net of normalised positive-definite functions in $B_\lambda(\GG)$ which would be ``eventually separated from $0$''. To make this statement precise, we use an auxiliary strictly dense subspace $A\subset \mrm{C}_b(\GG)$.

\begin{tw}\label{thm:Godementnv}
Let $\GG$ be a locally compact quantum group, $\emptyset\neq S$ a subset of $\Rep(\QG)$ and $A\subset \mrm{C}_b(\GG)$ a strictly dense subspace. If  there exists $\eps>0$ and a net $(f_i)_{i\in I}$ of normalised positive definite functions in $B_S(\GG)$ such that
\begin{equation}\label{eq5}
\forall_{a\in A} \exists_{i_0\in I}\forall_{i\ge i_0}\quad
a^* f_i a \ge \eps a^* a\;\textnormal{ in }\mrm{C}_b(\GG),
\end{equation}
then the Godement mean does not vanish on $B_S(\QG)$.
\end{tw}

\begin{proof}
Assume that there is such $\eps>0$ and a net $(f_i)_{i\in I}$. For each $i\in I$, write $f_i=(\id\otimes \omega_i\circ \pi_i) (\wW^*)$ for a state $\omega_i\in \B(\Hil_{i})_*$ and a representation $\pi_i$ weakly contained in $S$. Since $f_i$ are normalised, we can pass to a subnet $(\omega_j\circ\pi_j)_{j\in J}$ in $\mrm{C}_0^u(\whG)^*$ which converges weak$^*$ to a positive functional $\omega\in \mrm{C}_0^u(\whG)^*$. Let $(\zeta_\omega,\pi_\omega,\Hil_\omega)$ be the GNS representation for $\omega$. As in the proof of Proposition \ref{prop:approxpdf} we use \cite[Proposition 3.4.9]{DixmierC} to deduce that since all $\pi_j$ are weakly contained in $S$, the same is true for $\pi_\omega$. Consequently $b=(\id\otimes\omega)(\wW^*)\in B_S(\GG)$. Choose $a\in A$ and $\nu\in \LL^1(\GG)^+$. We have
\[\begin{split}
\la \nu , a^* b a \ra & =\la a\nu a^* , b \ra =
\la \omega , (a \nu a^*  \otimes \id)(\wW^*)\ra =
\lim_{j\in J} 
\la \omega_j\circ \pi_j , (a \nu a^*  \otimes \id)(\wW^*)\ra\\
&=
\lim_{j\in J} 
\la a \nu a^* , f_j\ra=
\lim_{j\in J} \la \nu , a^* f_j a \ra \ge 
\liminf_{j\in J} \la \nu , \eps a^* a \ra =
\la \nu , \eps a^* a \ra .
\end{split}\]
Since $\nu\ge 0$ is arbitrary, we obtain $a^* b a \ge \eps a^* a $ in $\LL^{\infty}(\GG)$, hence in $\mrm{C}_b(\GG)$. Since $a\in A$ was arbitrary, strict density of $A$ in $\mrm{C}_b(\GG)$ implies $b\ge \eps \I$. Indeed, take $\xi=c\xi_0$ for $c\in \mrm{C}_0(\GG), \xi_0\in \LL^2(\GG)$ and $(a_k)_{k\in K}$ a net in $A$ which converges strictly to $\I$. Then
\[
0\le \la \xi | (a_k^* b a_k - \eps a_k^* a_k) \xi \ra =
\la a_k c \xi_0 | (b-\eps \I) a_k c \xi_0 \ra \xrightarrow[k\in K]{} \la c \xi_0 | (b-\eps \I)c \xi_0 \ra .
\]
As this holds for all vectors $\xi=c\xi_0$ in a dense subset of $\LL^2(\GG)$, we can deduce that $b\ge \eps\I$.
Applying Godement mean $M$ to both sides of this inequality gives $M(b)\ge \eps$.
\end{proof}

The above theorem together with Proposition \ref{prop:Godement} immediately yields the following two corollaries.

\begin{cor}\label{cor1}
	Let $\GG$ be a locally compact quantum group and $A\subset \mrm{C}_b(\GG)$ a strictly dense subspace. If $\whG$ is not coamenable, then there is no $\eps>0$ and no net $(f_i)_{i\in I}$ of normalised positive definite functions in $B_\lambda(\GG)$ such that
	\begin{equation}\label{eq19}
		\forall_{a\in A} \exists_{i_0\in I}\forall_{i\ge i_0}\quad
		a^* f_i a \ge \eps a^* a\;\textnormal{ in }\mrm{C}_b(\GG).
	\end{equation}
\end{cor}

\begin{cor}\label{cor2}
	Let $\GG$ be a locally compact quantum group and $A\subset \mrm{C}_b(\GG)$ a strictly dense subspace. If $\QG$ does not have the Haagerup property then there is no $\eps>0$ and no net $(f_i)_{i\in I}$ of normalised positive definite functions in $B_0(\GG)$ such that
	\begin{equation}\label{eq20}
		\forall_{a\in A} \exists_{i_0\in I}\forall_{i\ge i_0}\quad
		a^* f_i a \ge \eps a^* a\;\textnormal{ in }\mrm{C}_b(\GG).
	\end{equation}
\end{cor}

In what follows we formulate certain consequences of the above corollaries.

\begin{cor}\label{cor3} Let $\QG$ be a locally compact quantum group, $\eps>0$ and $(f_i)_{i\in I}$ a net of normalised positive definite functions in $B_\lambda(\QG)$. Assume that one of the following is true:
	\begin{enumerate}
		\item $\GG$ is arbitrary and $A=\A(\QG)$ or $A=\{(\id\otimes\omega)(\ww^*)\,|\, \omega\in \LL^1_{\sharp}(\whG)\}$;
		\item $\GG$ is a discrete quantum group and $A$ is equal to $\mrm{c}_{c}(\GG)$ the algebraic direct sum $alg-\bigoplus_{\alpha\in \Irr(\whG)} \B(\Hil_\alpha)$;
		\item more generally, $\GG$ is an algebraic quantum group and $A=\mathfrak{C}_c^{\infty}(\GG)$.
\end{enumerate}
If the condition \eqref{eq19} holds, then $\whG$ is coamenable.
\end{cor}

Let us note that in cases $(2),(3)$, we can interpret condition \eqref{eq19} as an inequality ``$f_i\ge \eps \I$'' which holds pointwise-eventually (resp.~eventually uniformly on compact sets). A similar corollary is true in the case of Haagerup property.

Finally let us record the corresponding corollary for classical locally compact groups, whose `non-amenable' part can be also deduced from \cite{Derighetti}.

\begin{cor}\label{cor4} Let $G$ be a locally compact group. If $G$ is not amenable (respectively, does not have the Haagerup property) then for no $\eps>0$ can we find a net $(f_i)_{i \in I}$ of normalised positive-definite functions which are compactly supported (respectively, belong to $\C_0(G)$) and on every compact subset of $G$ are eventually greater than $\eps$. 
\end{cor}
\begin{proof}
	Recall first that, as already mentioned above, compactly supported positive-definite functions automatically belong to $\A(G) \subset B_\lambda(G)$ (as shown already in \cite{Eymard}) and positive-definite  functions in $\C_0(G)$ automatically belong to $B_0(G)$  by  \cite[Lemma 4.3]{QHap}. 
	
	We can apply Corollaries \ref{cor1} and \ref{cor2} with $A= \C_c(G)$. In the commutative case the inequality appearing in these corollaries amounts then to saying that given a supposed net of positive-definite functions $(f_i)_{i\in I}$,  for every compact set $Z \subset G$ we have $i_Z \in I$ such that for each $i\geq i_Z$ we have $f_i|_Z\geq \eps$. %As we are dealing with continuous functions and can afford to modify $\eps$ to $\eps/2$, a simple {\color{blue} (?)} compactness argument shows that we can replace above general compact subsets simply by points of $G$. We leave the details to the reader.
\end{proof}

Note that for $G$ discrete the last condition (being on every compact subset of $G$  eventually greater than $\eps$) means simply that for every point $t \in G$ we have $\limsup_{i \in I}f_i(t)\geq \eps$.

\section{Separation properties for von Neumann algebras of discrete quantum groups} \label{sec:sepOA}

In this section we let $\GG$ be a compact quantum group. If $\GG$ is coamenable, then the von Neumann algebra $\LL^{\infty}(\GG)$ is injective, i.e.~has $w^*$-CPAP (\cite[Theorem 1.1]{BMT}) and we can find a net $(\Phi_\lambda)_{\lambda\in \Lambda}$ of normal, finite rank UCP maps $\LL^{\infty}(\GG)\rightarrow\LL^{\infty}(\GG)$ which converges to identity in the point-$w^*$-topology. In fact, we can assume that maps $\Phi_\lambda$ are \emph{quantum Herz-Schur multipliers} (known also as \emph{adjoints of centralisers}), i.e.~they are given by $\Phi_\lambda=(\omega_\lambda\otimes \id)\Delta$ for some states $\omega_\lambda\in \LL^1(\GG)$. 

We will now exploit the separation results for quantum positive-definite functions obtained in the previous section to obtain their analogues on the von Neumann algebraic level. This will in particular strengthen the result mentioned above.

Before we formulate the first proposition, let us make some preliminary observations. Recall that by Lemma \ref{lem:dual} we have a linear bijection $B_\lambda(\whG)\simeq \mrm{C}(\GG)^*$ and any $a\in B_\lambda(\whG)$ can be written as $a=(\id\otimes \omega_{\xi,\eta})(U^{ *})$ for some unitary representation $U\lec \ww^{\whG}$ and vectors $\xi,\eta\in \Hil_U$. As $U\lec \ww^{\whG}$, functional $\omega_{\xi,\eta}\circ\phi_U\in \mrm{C}^u(\GG)^*$ is well defined  on $\mrm{C}^u(\GG)/\ker \lambda_{\GG}=\mrm{C}(\GG)$ and we can write $a=(\id\otimes\omega) (\ww^{\whG *})$ for some (not necesarilly normal) functional $\omega\in \mrm{C}(\GG)^*$ -- the image of $\omega_{\xi,\eta}\circ\phi_U$. Furthermore, as $B_\lambda(\whG)\subset \M^l_{cb}(\A(\whG))$, we can use the associated normal CB map $\Theta^l(a)\in\CB^\sigma(\LL^{\infty}(\GG))$. According to \cite[Proposition 4.8]{Brannan} (see also \cite{DawsMultipliers}), this map is given by
\[
\Theta^l(a)\colon \LL^{\infty}(\GG)\ni x \mapsto
(\id\otimes \omega_{\xi,\eta}) (U (x\otimes \I)U^*)\in\LL^{\infty}(\GG).
\]

Combining these two properties gives us 
\begin{equation}\label{eq1}
	\Theta^l(a)\colon \LL^{\infty}(\GG)\ni x \mapsto
	(\id\otimes \omega) (\ww^{\whG} (x\otimes \I)\ww^{\whG *})\in\LL^{\infty}(\GG),
\end{equation}
where we interpret $\ww^{\whG},\ww^{\whG *}$ as elements of $\M(\mc{K}(\LL^2(\GG))\otimes \mrm{C}(\GG))$, so that $\ww^{\whG} (x\otimes \I)\ww^{\whG *}\in \M(\mc{K}(\LL^2(\GG))\otimes \mrm{C}(\GG))$ and \eqref{eq1} makes sense.

Here and below $\|\cdot\|_2$ denotes the $\LL^2$-norm on $\LL^{\infty}(\GG)$ induced by the Haar integral, i.e.~$\|x\|_2=h(x^*x)^{1/2}$, $x \in \LL^\infty(\GG)$. One can view the next proposition (and most results in this section) as a way of separating finite rank maps from the identity using $x\in \LL^{\infty}(\GG)$  and $\omega \in \LL^{1}(\GG)$.

\begin{prop}\label{prop5}
	Let $\GG$ be a compact quantum group which is not coamenable and let $\eps\in (0,1)$. Suppose that $(\Phi_\lambda)_{\lambda \in \Lambda}$ is a net of normal, UCP quantum Herz-Schur multipliers $\Phi_\lambda \colon \LL^{\infty}(\GG)\rightarrow\LL^{\infty}(\GG)$ , given by $\Phi_\lambda=\Theta^l(a_\lambda)$ for $a_\lambda\in B_\lambda(\whG)$. Then there exist $x\in \Pol(\GG)$ and $\omega\in \LL^1(\GG)$  with $\|x\|_2=\|\omega\|=1$ such that $\limsup_{\lambda\in \Lambda}|\langle x-\Phi_\lambda(x),\omega\rangle| > \eps$.
\end{prop}

\begin{proof}
	For each $\alpha\in \Irr(\GG)$ choose a representative $U^{\alpha}\in \alpha$ and an orthonormal basis $\{\xi^\alpha_{i}\}_{i=1}^{\dim(\alpha)}$ in $\msf{H}_\alpha$. Assume furthermore that $\uprho_\alpha$-operators (see \cite[Section 1.7]{NeshveyevTuset}) are diagonal with respect to the chosen basis, with eigenvalues $\{\uprho_{\alpha,i}\}_{i=1}^{\dim(\alpha)}$.
	
	Assume by contradiction that for all $x\in\Pol(\GG),\rho\in \LL^1(\GG)$ with $\|x\|_2=\|\rho\|=1$ we have 
\begin{equation}\label{eq21}
\limsup_{\lambda\in \Lambda}|\la x-\Phi_\lambda(x),\rho\ra | \le \eps.
\end{equation}
Fix $\lambda \in \Lambda$. As recalled above (equation \eqref{eq1}) the map $\Phi_\lambda=\Theta^l(a_\lambda)$ is given by 
	\[
	\Phi_\lambda(x)=
	(\id\otimes \omega_{\lambda}) (\ww^{\whG} (x\otimes \I)\ww^{\whG *})
	\]
	for $x\in\LL^{\infty}(\GG)$ and some functionals $\omega_\lambda\in \mrm{C}(\GG)^*$. As $\Phi_\lambda$ is UCP, $\omega_\lambda$ is a state. Consequently
	\begin{equation}\label{eq10}
		\la \Phi_\lambda(x), \rho \ra =
		\omega_{\lambda}\bigl(
		(\rho\otimes \id) (\ww^{\whG} (x\otimes \I)\ww^{\whG *})
		\bigr)\quad(\rho\in \LL^1(\GG))
	\end{equation}
	and $(\rho\otimes \id) (\ww^{\whG} (x\otimes \I)\ww^{\whG *})
	\in \M(\mrm{C}(\GG))=\mrm{C}(\GG)$. For each $\lambda\in\Lambda$, we can find a net $(\omega_{\lambda,i})_{i\in I}$ of normal states in $\LL^1(\GG)$ such that $\omega_{\lambda,i}\xrightarrow[i\in I]{}\omega_\lambda$ pointwise on $\mrm{C}(\GG)$. Using the fact that $\LL^{\infty}(\GG)\subset \B(\LL^2(\GG))$ is standard and approximating further, we can assume that $\omega_{\lambda,i}=\omega_{\xi_{\lambda,i}}$ for length $1$ vectors $\xi_{\lambda,i}\in \Lambda_h(\Pol(\GG))$. Define $\Phi_{\lambda,i}=(\omega_{\lambda,i}\otimes \id)\Delta$. Observe using \eqref{eq10} that for any $x\in \LL^{\infty}(\GG),\rho\in\LL^1(\GG)$ % and $\lambda\in \Lambda$
	\begin{align*}
		\lim_{i\in  I}
		\la x - \Phi_{\lambda,i}(x) , \rho \ra &=
		\lim_{i\in  I}
		\la x - (\omega_{\lambda,i}\otimes \id)\Delta(x) , \rho \ra =
		\la x,\rho\ra  -
		\lim_{i\in I} \la  \omega_{\lambda,i}, 
		(\rho\otimes \id)(
		\ww^{\whG}(x\otimes \I)\ww^{\whG *}
		) \ra \\
		&=
		\la x,\rho\ra  -
		\la  \omega_{\lambda}, 
		(\rho\otimes \id)(
		\ww^{\whG}(x\otimes \I)\ww^{\whG *}
		) \ra =
		\la x -\Phi_\lambda(x) , \rho \ra .
	\end{align*}
	
	Consequently (after passing to a new net), we obtain a net of finite rank UCP maps $(\wt{\Phi}_\lambda)_{\lambda\in \Lambda}$ which still satisfies \eqref{eq21} for all $x\in \Pol(\QG),\rho\in\LL^1(\QG)$ with $\|x\|_2=\|\rho\|=1$ with each $\wt{\Phi}_{\lambda}$ given by $\wt{\Phi}_\lambda=\Theta^l(\wt{a}_\lambda)$ where $\wt{a}_\lambda=(\id\otimes \wt{\omega}_{\lambda})(\ww^{\whG *})$ for some $\wt{\omega}_{\lambda}=\omega_{\xi_\lambda}$, $\xi_\lambda\in \Lambda_h(\Pol(\GG))$. Explicitly, the map $\wt{\Phi}_\lambda$ is given by $\wt{\Phi}_\lambda=(\wt{\omega}_\lambda\otimes \id)\Delta$. Next we need to correct functionals $\wt{\omega}_\lambda$.\\
	
	Pick $m_{\RR}\in \LL^\infty(\RR)^*$, a mean on $\RR$, and define (c.f.~\cite{Tomatsu})
	\[
	\omega_\lambda^{(1)}\colon \LL^{\infty}(\GG)\ni y\mapsto m_{\RR}\bigl(\RR\ni t \mapsto 
	\wt{\omega}_\lambda( \tau_t^{\GG}(y))\in \CC\bigr)\in \CC\quad(\lambda\in \Lambda),
	\]
	where $(\tau_t)_{t \in \mathbb{R}}$ denotes the \emph{scaling automorphism group} of $\LL^\infty(\QG)$ -- see \cite[Section 1.7]{NeshveyevTuset}.
	Since $\wt{\Phi}_\lambda$ is finite rank, acts via $\wt{\Phi}_\lambda(U^{\alpha}_{i,j})=\sum_{k=1}^{\dim(\alpha)} \wt{\omega}_\lambda(U^{\alpha}_{i,k}) U^{\alpha}_{k,j}$ and $U^{\alpha}_{i,j}$'s form a linearly independent set, we can conclude that there is a finite set $F_\lambda\subset \Irr(\GG)$ so that $\omega_\lambda(U^{\alpha}_{i,j})=0$ for all $\alpha\in F_\lambda^c= \Irr(\GG)\setminus F_\lambda$, $1\le i,j\le \dim(\alpha)$. Enlarge $F_\lambda$, so that it is closed under taking the contragradient representation. Next, as $\tau_t^{\GG}(U^{\alpha}_{i,j})=\bigl(\tfrac{\uprho_{\alpha,i}}{\uprho_{\alpha,j}}\bigr)^{it} U^{\alpha}_{i,j}$ we still have $\omega_\lambda^{(1)}(U^{\alpha}_{i,j})=0$ for $\alpha\in F_\lambda^c$ and consequently $\omega_\lambda^{(1)}$ is a normal state. Finally, use the \emph{unitary antipode} $R_{\GG}$ (\cite[Proposition 1.7.11]{NeshveyevTuset}) to define normal states
	\[
	\omega_\lambda^{(2)}=\tfrac{1}{2}(\omega_\lambda^{(1)}+ \omega_\lambda^{(1)}\circ R_{\GG})\in \LL^1(\GG)
	\]
	(observe that $\omega_\lambda^{(1)}\circ R_{\GG}$ is a normal state since $R_{\GG}$ is positive and normal) and normalised positive definite functions
	\[
	f_\lambda=(\id\otimes \omega_\lambda^{(2)})(\ww^{\whG *}).
	\]
	Since also $\omega_\lambda^{(2)}(U^{\alpha}_{i,j})=0$ whenever $\alpha\in F_\lambda^c$, we have $f_\lambda\in \mrm{c}_{c}(\whG)$. Let us argue that $f_\lambda$ are self-adjoint. Observe first that since $\omega_\lambda^{(1)}\circ\tau^{\GG}_t=\omega_\lambda^{(1)}$, $(R_{\whG}\otimes R_{\GG})(\ww^{\whG})=\ww^{\whG }$ and $(\tau_t^{\whG}\otimes \tau_t^{\GG})(\ww^{\whG })=\ww^{\whG }$ for all $t\in \RR$ (see \cite[Section 5]{VanDaele}) we have
	\begin{align*}
		(\id\otimes \omega_\lambda^{(1)})(\ww^{\whG *})& =
		S_{\whG}\bigl((\id\otimes \omega_\lambda^{(1)})(\ww^{\whG})\bigr)=
		R_{\whG} \tau^{\whG}_{-i/2}\bigl(
		(\id\otimes \omega_\lambda^{(1)})(\ww^{\whG})\bigr)\\
		&=
		R_{\whG} \bigl(
		(\id\otimes \omega_\lambda^{(1)})(\ww^{\whG})\bigr)=
		(\id\otimes \omega_\lambda^{(1)}\circ R_{\GG})(\ww^{\whG}).
	\end{align*}
	Consequently
	\[\begin{split}
		f_\lambda&=\tfrac{1}{2}\bigl( (\id\otimes \omega_\lambda^{(1)})(\ww^{\whG *})+
		(\id\otimes \omega_{\lambda}^{(1)}\circ R_{\GG})(\ww^{\whG *})\bigr)\\
		&=
		\tfrac{1}{2} \bigl( 
		(\id\otimes \omega_\lambda^{(1)}\circ R_{\GG})(\ww^{\whG})+
		(\id\otimes \omega_\lambda^{(1)}\circ R_{\GG})(\ww^{\whG})^*\bigr)
	\end{split}\]
	is self-adjoint.
	
	Since $0<\eps<1$, we can choose $\eps',\eps''$ satisfying $\eps<\eps'<\eps''<1$. To obtain a contradiction, we will use point $(2)$ of Corollary \ref{cor3}. Namely, let us argue that for any $a\in \mrm{c}_{c}(\whG)$ there is $\lambda_0\in \Lambda$ so that for $\lambda\ge \lambda_0$
	\begin{equation}\label{eq2}
		a^* f_\lambda a \ge (1-\eps'') a^* a.
	\end{equation}
	This will give us a contradiction. First observe that to get \eqref{eq2} it is enough to consider central projections $a=p_\alpha$ for $\alpha\in \Irr(\GG)$. Indeed, write $a=\sum_{\alpha\in F} a p_\alpha$ for a finite set $F\subset \Irr(\GG)$ and assume that \eqref{eq2} holds for all $p_\alpha\,(\alpha\in F)$ and corresponding $\lambda_\alpha$. Choose $\lambda_0 \in \Lambda$ such that $\lambda_0\ge \lambda_\alpha\,(\alpha\in F)$. Then for $\lambda\ge \lambda_0$ we have
	\[
	a^* f_\lambda a=
	\sum_{\alpha\in F} (a^* p_\alpha) (p_\alpha^* f_\lambda p_\alpha) (a p_\alpha) \ge 
	\sum_{\alpha\in F} (a^* p_\alpha) ((1-\eps'') p_\alpha^* p_\alpha) (a p_\alpha)=
	(1-\eps'') a^* a.
	\]
	
	Thus it is enough to prove \eqref{eq2} for $a=p_\alpha$. Fix $\alpha\in \Irr(\GG)$, non-zero vectors $\xi,\eta,\zeta \in \Hil_\alpha$ with $\|\eta\|=1$ and consider elements $x=\tfrac{U^{\alpha}_{\zeta,\eta }}{\|U^{\alpha}_{\zeta,\eta}\|_2}, \rho=\tfrac{h(U^{\alpha *}_{\xi,\eta }\cdot)}{\|h(U^{\alpha *}_{\xi,\eta}\cdot)\|}$. The assumption \eqref{eq21} (or rather its variant for the modified maps $\wt{\Phi}_\lambda$) implies in particular that
	\begin{equation}\label{eq11}
		\limsup_{\lambda\in \Lambda} \bigl|h\bigl( U^{\alpha *}_{\xi,\eta } (U^{\alpha}_{\zeta,\eta }  - \wt{\Phi}_\lambda(U^{\alpha}_{\zeta,\eta }))\bigr)  \bigr| \le  \|U^{\alpha}_{\zeta,\eta }\|_2 \|h(U^{\alpha *}_{\xi,\eta}\cdot)\|\,\eps,
	\end{equation}
	i.e.\ there exists $\lambda_0\in \Lambda$ (depending on $\xi,\eta,\zeta$) so that for $\lambda\ge \lambda_0$
	\[
	\bigl|h\bigl( U^{\alpha *}_{\xi,\eta } (U^{\alpha}_{\zeta,\eta}  - \wt{\Phi}_\lambda(U^{\alpha}_{\zeta,\eta}))\bigr)  \bigr| \le  \|U^{\alpha}_{\zeta,\eta}\|_2 \|h(U^{\alpha *}_{\xi,\eta}\cdot)\|\,\eps'.
	\]
	The left hand side of the above inequality is equal to
	\[\begin{split}
	\bigl|h\bigl( U^{\alpha *}_{\xi,\eta } (U^{\alpha}_{\zeta,\eta}  - \wt{\Phi}_\lambda(U^{\alpha}_{\zeta,\eta}))\bigr)  \bigr|&=
	\bigl|\tfrac{\la \zeta |\uprho_{\alpha}^{-1}\xi\ra  }{\dim_q(\alpha)} - \sum_{m=1}^{\dim(\alpha)}
	\omega_\lambda(U^{\alpha}_{\zeta,\xi^{\alpha}_m} )
	\tfrac{ \la \xi^\alpha_m | \uprho_\alpha^{-1} \xi\ra  }{ \dim_q(\alpha)}\bigr|\\
	&=
	\tfrac{1}{\dim_q(\alpha)}
	\bigl|\la \zeta|\uprho_\alpha^{-1}\xi\ra - \omega_\lambda(U^{\alpha}_{\zeta,\uprho_\alpha^{-1}\xi})\bigr|.
	\end{split}\]
	Hence we obtain
	\begin{equation}\label{eq3}
		\bigl|\la \zeta|\uprho_\alpha^{-1}\xi\ra - \omega_\lambda(U^{\alpha}_{\zeta,\uprho_\alpha^{-1}\xi})\bigr|
		\le  \dim_q(\alpha) \|U^{\alpha}_{\zeta,\eta } \|_2 \|h(U^{\alpha *}_{\xi,\eta}\cdot)\|\,\eps' \le 
		\eps' \sqrt{ \la \xi|\uprho_\alpha^{-1}\xi\ra \, \la \zeta | \uprho_\alpha^{-1}\zeta \ra }
	\end{equation}
	for fixed $\xi,\zeta$ and all $\lambda\ge \lambda_0$. Next we obtain a bound for $\omega^{(1)}_\lambda$. Assume additionally that $\xi,\zeta\in 1_{\{c\}}(\uprho_\alpha)\msf{H}_\alpha$ for some $c\in\operatorname{Sp}(\uprho_\alpha)$. Recall that $\tau^{\GG}_t(U^{\alpha}_{\xi,\eta})=U^{\alpha}_{\uprho_{\alpha}^{-it} \xi,\uprho_{\alpha}^{-it}\eta}$, $t \in \br$. Consequently using the inequality \eqref{eq3} we have
	\begin{equation}\begin{split}\label{eq44}
			\bigl|\la \zeta|\uprho_\alpha^{-1}\xi\ra - \omega_\lambda^{(1)}(U^{\alpha}_{\zeta,\uprho_\alpha^{-1}\xi})\bigr|&=
			\bigl|m_{\RR}\bigl(\RR\ni t \mapsto 
			\la  \zeta|\uprho_\alpha^{-1}\xi\ra - \omega_\lambda(U^{\alpha}_{\uprho_{\alpha}^{-it} \zeta,\uprho_\alpha^{-1-it }\xi})\in \CC\bigr)\bigr|\\
			&=
			\bigl|m_{\RR}\bigl(\RR\ni t \mapsto 
			\la \zeta|\uprho_\alpha^{-1}\xi\ra - \omega_\lambda(U^{\alpha}_{\zeta,\uprho_\alpha^{-1}\xi})\in \CC\bigr)\bigr|\\
			&=
			\bigl|\la \zeta|\uprho_\alpha^{-1}\xi\ra - \omega_\lambda(U^{\alpha}_{\zeta,\uprho_\alpha^{-1}\xi})\bigr|\le 
			\eps' \sqrt{\la \xi | \uprho_\alpha^{-1}\xi \ra \la \zeta | \uprho_\alpha^{-1} \zeta\ra }
	\end{split}\end{equation}
	for $\xi,\zeta$ in the same eigenspace of $\uprho_\alpha$ and all $\lambda\ge \lambda_0$. The next step is to obtain a bound for $\omega^{(2)}_\lambda$. Let again $\xi,\zeta\in 1_{\{c\}}(\uprho_\alpha)\msf{H}_\alpha$ for a fixed $c\in\operatorname{Sp}(\uprho_\alpha)$. Using \eqref{eq44} and the fact that $\omega^{(1)}_{\lambda}$ is a state invariant under scaling group we deduce that
	\begin{equation}\label{eq18}
		\begin{split}
			\bigl|\la \zeta|\uprho_\alpha^{-1}\xi\ra - \omega_\lambda^{(2)}(U^{\alpha}_{\zeta,\uprho_\alpha^{-1}\xi})\bigr|&\le 
			\tfrac{1}{2}\bigl(
			\bigl|\la \zeta|\uprho_\alpha^{-1}\xi\ra - \omega_\lambda^{(1)}(U^{\alpha}_{\zeta,\uprho_\alpha^{-1}\xi})\bigr|+
			\bigl|\la \zeta|\uprho_\alpha^{-1}\xi\ra - \omega_\lambda^{(1)}(R_{\GG}(U^{\alpha}_{\zeta,\uprho_\alpha^{-1}\xi}))\bigr|
			\bigr)\\
			&=
			\tfrac{1}{2}\bigl(
			\bigl|\la \zeta|\uprho_\alpha^{-1}\xi\ra - \omega_\lambda^{(1)}(U^{\alpha}_{\zeta,\uprho_\alpha^{-1}\xi})\bigr|+
			\bigl|\la \zeta|\uprho_\alpha^{-1}\xi\ra - \omega_\lambda^{(1)}(U^{\alpha* }_{\uprho_\alpha^{-1}\xi,\zeta})\bigr|
			\bigr)\\
			&=
			\tfrac{1}{2}\bigl(
			\bigl|\la \zeta|\uprho_\alpha^{-1}\xi\ra - \omega_\lambda^{(1)}(U^{\alpha}_{\zeta,\uprho_\alpha^{-1}\xi})\bigr|+
			\bigl|\la \uprho_\alpha^{-1}\xi|\zeta\ra - \omega_\lambda^{(1)}(U^{\alpha }_{\uprho_\alpha^{-1}\xi,\zeta})\bigr|
			\bigr)\\
			&=
			\tfrac{1}{2}\bigl(
			\bigl|\la \zeta|\uprho_\alpha^{-1}\xi\ra - \omega_\lambda^{(1)}(U^{\alpha}_{\zeta,\uprho_\alpha^{-1}\xi})\bigr|+
			\bigl|\la \xi|\uprho_\alpha^{-1}\zeta\ra - \omega_\lambda^{(1)}(U^{\alpha }_{\xi,\uprho_\alpha^{-1}\zeta})\bigr|
			\bigr)\\
			&\le 
			\eps' \sqrt{\la \xi | \uprho_\alpha^{-1}\xi \ra \la \zeta | \uprho_\alpha^{-1} \zeta\ra }.
		\end{split}
	\end{equation}
	Now, pick $0<\delta<1$ so that $6\delta+\eps'\le \eps''$. For a fixed $c\in \operatorname{Sp}(\uprho_\alpha)$ choose a finite set $\{\theta_k\}_{k=1}^{N}$ in the sphere of $1_{\{c\}}(\uprho_\alpha)\msf{H}_\alpha$ which forms a $\tfrac{\delta}{\|\uprho_\alpha\|\|\uprho_\alpha^{-1}\|}$-net. Since $N<+\infty$, iterating \eqref{eq18} we can find $\lambda_1\ge \lambda_0$ (depending on $c $) so that for all $1\le k,k'\le N, \lambda\ge \lambda_1$ we have
	\begin{equation}
		\bigl|\la \theta_k|\uprho_\alpha^{-1}\theta_{k'}\ra - \omega_\lambda^{(2)}(U^{\alpha}_{\theta_k,\uprho_\alpha^{-1}\theta_{k'}})\bigr|\le 
		\eps' \sqrt{\la \theta_k | \uprho_\alpha^{-1}\theta_k \ra \la \theta_{k'} | \uprho_\alpha^{-1} \theta_{k'}\ra }.
	\end{equation}
	Now choose any norm $1$ vector $\theta\in 1_{\{c\}}(\uprho_\alpha) \msf{H}_\alpha$ and $1\le k\le N$ so that $\|\theta-\theta_k\|\le \tfrac{\delta}{\|\uprho_\alpha\|\|\uprho_\alpha^{-1}\|}$. We obtain
	\begin{equation}\begin{split}\label{eq7}
			\bigl|\la \theta|\uprho_\alpha^{-1}\theta\ra - \omega_\lambda^{(2)}(U^{\alpha}_{\theta,\uprho_\alpha^{-1}\theta})\bigr|&\le 
			4\|\theta-\theta_k\| \|\uprho_\alpha^{-1}\|+
			\bigl|\la \theta_k|\uprho_\alpha^{-1}\theta_{k}\ra - \omega_\lambda^{(2)}(U^{\alpha}_{\theta_k,\uprho_\alpha^{-1}\theta_{k}})\bigr|\\
			&\le 
			4\tfrac{\delta}{\|\uprho_\alpha\|}+\eps' \la \theta_k | \uprho_\alpha^{-1}\theta_k\ra \le 
			4\tfrac{\delta}{\|\uprho_\alpha\|}+\eps' ( 2\tfrac{\delta}{\|\uprho_\alpha\|}+ \la \theta|\uprho_\alpha^{-1} \theta \ra )\\
			&\le 6\tfrac{\delta}{\|\uprho_\alpha\|} + \eps' \la \theta|\uprho_\alpha^{-1}\theta \ra.
	\end{split}\end{equation}
	Inequality \eqref{eq7} holds for fixed $c$, but as these come from a finite set $\oon{Sp}(\uprho_\alpha)$, we can find $\lambda_2 \in \Lambda$ such that the above inequality holds for all $c\in \operatorname{Sp}(\uprho_\alpha)$, $\lambda \geq \lambda_2$.
	
	We already proved that $f_\lambda p_\alpha=(\id\otimes \omega_\lambda^{(2)})(\ww^{\whG *}) p_\alpha$ is self-adjoint. Now we claim that 
	\[
	f_\lambda p_\alpha \ge (1-\eps'') p_\alpha.
	\]
	We can think of both operators as acting on $\msf{H}_\alpha$. Observe that the operator $f_\lambda p_\alpha$ is block-diagonal, i.e.~for $c\in \operatorname{Sp}(\uprho_\alpha)$ we have $(f_\lambda p_\alpha) 1_{\{c\}}(\uprho_\alpha)=1_{\{c\}}(\uprho_\alpha) (f_\lambda p_\alpha)$ -- this follows from the fact that $f_\lambda$ is invariant under the scaling group. Let $\mu\in \RR$ be an eigenvalue of $f_\lambda p_\alpha$. Then $\mu\le 1$. Since $f_\lambda p_\alpha$ is block-diagonal, we can find a corresponding eigenvector $\theta\in \msf{H}_\alpha$ of norm $1$ which is included in some $1_{\{c\}}(\uprho_\alpha)\msf{H}_\alpha$. Then
	\[
	\mu=\la \theta | (f_\lambda p_\alpha) \theta \ra =
	(\omega_\theta\otimes \omega_{\lambda}^{(2)})(\ww^{\whG *})=
	(\omega_{\lambda}^{(2)}\otimes \omega_\theta)(\ww^{\GG})=
	\omega_\lambda^{(2)}(U^{\alpha}_{\theta,\theta}),
	\] 
	hence using \eqref{eq7}
	\[
	|\mu - 1 | =
	|\omega_\lambda^{(2)}(U^{\alpha}_{\theta,\theta})-\la \theta | \theta \ra |=c
	|\omega_\lambda^{(2)}(U^{\alpha}_{\theta,\uprho_{\alpha}^{-1}\theta})-\la \theta | \uprho_{\alpha}^{-1}\theta \ra | \le 
	c\bigl(6\tfrac{\delta}{\|\uprho_\alpha\|} + \eps' \la \theta |\uprho_{\alpha}^{-1} \theta \ra \bigr)\le 6\delta  + \eps'\le \eps''. 
	\]
	Consequently for $\lambda\ge\lambda_2$
	\[
	f_\lambda p_\alpha\textnormal{ is self-adjoint},\quad\operatorname{Sp}(f_\lambda p_\alpha)\subset [1-\eps'',1]\quad\Rightarrow\quad f_\lambda p_\alpha \ge (1-\eps'') p_\alpha.
	\]
	As $0<\eps''<1$ and $f_\lambda$ are normalised and positive definite, this gives us a contradiction.
\end{proof}

Proposition \ref{prop5} holds in the full generality of (possibly non-Kac type) compact quantum groups. Its downside is however that it does not say anything about $\LL^{\infty}(\GG)$ as a von Neumann algebra if we forget about the quantum group structure, as we had to assume that maps $\Phi_\lambda$ are associated to multipliers from the space $B_\lambda(\whG)$. Another downside is that we have control only over the $\LL^2$-norm of the operator $x$ in the statement, not the operator norm. In the following results we will prove different separating results, which are  similar in spirit, but where these problems will be remedied --  under additional assumptions. First let us introduce some convenient terminology.

Recall that for any von Neumann algebra $\M$ we have a canonical completely isometric identification $\CB(\M)=(\M_*\widehat{\otimes}\M)^*$, where $\widehat{\otimes}$ is the projective tensor product of operator spaces (\cite[Corollary 7.1.5]{EffrosRuan}).

\begin{deft}\label{def1}
	Let $\M$ be a von Neumann algebra and  $\eps\in (0,1)$.
	\begin{itemize}
		\item We say that $\M$ has the \emph{$\eps$-separation property} if for every net $(\Phi_\lambda)_{\lambda\in \Lambda}$ of normal, finite rank, UCP maps on $\M$ there is $x\in \M$ and $\omega\in \M_*$ with $\|x\|=\|\omega\|=1$ such that $\limsup_{\lambda\in \Lambda} |\la x-\Phi_\lambda(x),\omega\ra|> \eps$.
		\item We say that $\M$ has the \emph{matrix $\eps$-separation property} if for every net $(\Phi_\lambda)_{\lambda\in \Lambda}$ of normal, finite rank, UCP maps on $\M$ there is a Hilbert space $\Hil$ and $x\in \M\bar{\otimes }\B(\Hil),\omega\in (\M\bar{\otimes} \B(\Hil))_*$ with $\|x\|=\|\omega\|=1$ such that $\limsup_{\lambda\in \Lambda} |\la x-(\Phi_\lambda\otimes \id)(x),\omega\ra|> \eps$.
%		\item We say that $\M$ has the \emph{operator $\eps$-separation property} if for every net $(\Phi_\lambda)_{\lambda\in \Lambda}$ of normal, finite rank, UCP maps on $\M$ there is $\Omega\in \M_*\widehat{\otimes} \M$ with $\|\Omega\|=1$ such that $\limsup_{\lambda\in \Lambda}|\langle \id-\Phi_\lambda , \Omega\rangle|>\eps$.
	\end{itemize}
	%Let us also introduce \emph{factorisable $\eps$-separation property} and \emph{factorisable matrix $\eps$-separation property} by additionally requiring that maps $\Phi_\lambda$ admit factorisation through matrix algebras, i.e.~$\Phi_\lambda\colon \M\xrightarrow[]{\phi_{\lambda}}\M_{n(\lambda)}\xrightarrow[]{\psi_{\lambda}}\M$ with $\phi_\lambda,\psi_\lambda$ normal UCP.
\end{deft}

We begin with some easy observations.

\begin{rem}\label{rem:sep}
	Fix $\eps \in (0,1)$.
 If $\M$ has the $\eps$-separation property, then it has the $\eps'$-separation property for all $0<\eps'<\eps$. The analogous statement holds for the matrix $\eps$-separation property. % and operator $\eps$-separation property. 
 Further, the matrix $\eps$-separation property, formally weaker than the  $\eps$-separation property, implies that $\M$ does not have $w^*$-CPAP.
\end{rem}

The next lemma shows that in the definition of the matrix $\eps$-separation property one can replace $\B(\Hil)$ by matrices (or arbitrary von Neumann algebras), which justifies the proposed terminology. Furthermore, we provide an equivalent formulation of this property which does not refer to any additional von Neumann algebra.

\begin{lem}
	Let $\M$ be a von Neumann algebra, $\eps\in (0,1)$ and $(\Phi_\lambda)_{\lambda\in \Lambda}$ a net of normal UCP maps on $\M$. The following are equivalent:
	\begin{enumerate}
		\item there is a natural number $n\in\NN$, $x\in \M\otimes \M_n$ and $\omega\in (\M\otimes \M_n)_*$ with $\|x\|=\|\omega\|=1$ such that $\limsup_{\lambda\in \Lambda} |\la x-(\Phi_\lambda\otimes \id)(x),\omega\ra|> \eps$;
		\item there is a Hilbert space $\Hil$, $x\in \M\bar{\otimes }\B(\Hil)$ and $\omega\in (\M\bar{\otimes} \B(\Hil))_*$ with $\|x\|=\|\omega\|=1$ such that $\limsup_{\lambda\in \Lambda} |\la x-(\Phi_\lambda\otimes \id)(x),\omega\ra|> \eps$;
		\item there is a von Neumann algebra $\N$, $x\in \M\bar{\otimes }\N$ and $\omega\in (\M\bar{\otimes} \N)_*$ with $\|x\|=\|\omega\|=1$ such that $\limsup_{\lambda\in \Lambda} |\la x-(\Phi_\lambda\otimes \id)(x),\omega\ra|> \eps$;
		\item there is $\Omega\in \M_*\widehat{\otimes} \M$ with $\|\Omega\|=1$ such that $\limsup_{\lambda\in \Lambda} |\la \id- \Phi_\lambda, \Omega\ra|> \eps$.
	\end{enumerate}
Consequently, $\M$ has the matrix $\eps$-separation property if, and only if for every net $(\Phi_\lambda)_{\lambda\in \Lambda}$ of normal, finite rank, UCP maps on $\M$ there is $\Omega\in \M_*\widehat{\otimes} \M$ with $\|\Omega\|=1$ such that $\limsup_{\lambda\in \Lambda}|\langle \id-\Phi_\lambda , \Omega\rangle|>\eps$.
\end{lem}

\begin{proof}[Sketch of a proof]
	Implications $(1)\Rightarrow (2)\Rightarrow (3)$ are trivial. To see that $(3)$ implies $(1)$, observe first that it is enough to consider $\N=\B(\Hil)$, as normal functionals can be extended to superalgebras without increasing norms. Next, use the identification $(\M\bar{\otimes}\B(\Hil))_*=\M_*\widehat{\otimes} \B(\Hil)_*$, where $\widehat{\otimes}$ is the projective operator space tensor product (see \cite[Theorem 7.2.4]{EffrosRuan}) and approximate $\omega$ by a finite sum of simple tensors. To finish the proof, use the fact that any normal functional on $\B(\Hil)$ is given by a $\Tr(T\cdot)$ for a trace class operator $T$ and finite rank operators are dense in the space of trace class operators.	

Next we show that $(1)\Rightarrow(4)\Rightarrow(2)$. 

Assume that $(1)$ holds with $x\in \M\otimes\M_n,\omega\in (\M\otimes\M_n)_*$. Define $\Omega_{\omega,x}\in \CB(\M)^*$ via $\langle \Omega_{\omega,x},T\rangle = \langle (T\otimes \id)(x),\omega\rangle $, $T \in \CB(\M)$.  The functional $\Omega_{\omega,x}$ cannot be zero, as then we would have 
$\la x-(\Phi_\lambda\otimes \id)(x),\omega\ra =0$ for each $\lambda \in \Lambda$. On the other hand $\|\Omega_{\omega,x}\|\le \|x\|\|\omega\|=1$.  Since $\M_n$ is finite dimensional, we can write $\omega=\sum_{i=1}^{N} \omega_{1,i}\otimes \omega_{2,i}$ for some $\omega_{1,i}\in \M_*, \omega_{2,i}\in \M_n^*$, $1\le i \le N$. We have
	\[
	\langle \Omega_{\omega,x} , T \rangle 
	=\sum_{i=1}^{N} \langle (T\otimes \id)(x) , \omega_{1,i}\otimes \omega_{2,i}\rangle = 
	\sum_{i=1}^{N} \langle T( (\id\otimes \omega_{2,i})(x)) , 
	\omega_{1,i}\rangle \quad(T\in \CB(\M)),
	\]
which shows that $\Omega_{\omega,x}\in \M_*\odot \M\subset \M_*\widehat{\otimes} \M$. To finish this part of proof, define $\Omega=\tfrac{\Omega_{\omega,x}}{\|\Omega_{\omega,x}\|}$.

Implication $(4)\Rightarrow(2)$ follows closely proof of \cite[Proposition 3.9]{DKV}, hence we provide only a sketch. Assume that we are given $\Omega\in \M_*\widehat{\otimes}\M$ as in $(4)$, and let $\eps'$ be such that $\limsup_{\lambda\in \Lambda} |\langle \id-\Phi_\lambda,\Omega\rangle |>\eps'>\eps$. According to \cite[Theorem 10.2.1]{EffrosRuan} we can find infinite matrices $\alpha\in\M_{1,\infty\times\infty}, \beta\in \K_\infty(\M_*),\gamma\in\K_\infty(\M),\alpha'\in \M_{\infty\times\infty,1}$ such that ${\Omega=\alpha (\beta\otimes \gamma)\alpha'}$ and $\|\alpha\|\|\beta\|\|\gamma\|\|\alpha'\|\le \tfrac{\eps'}{\eps}$. Write these matrices as $\alpha=[\alpha_{1,(i,j)}]_{(i,j)\in \NN^{\times 2}}$, etc.~(so that $\Omega=\sum_{i,j,k,l=1}^{\infty} \alpha_{1,(i,j)} (\beta_{i,k}\otimes \gamma_{j,l}) {\alpha'}_{(k,l),1}$) and let $e_{i,j}\,(i,j\in\NN)$ be the matrix units in $\B(\ell^2)$. One can check that $[e_{j,i}]_{i,j=1}^{\infty}$ is a well defined infinite matrix of norm $1$ in $\M_\infty(T(\ell^2))$, where $T(\ell^2)\simeq \B(\ell^2)_*$ is the space of trace class operators. Notice that $\gamma\in \K_\infty(\M)=\M\otimes \K_\infty\subset \M\bar{\otimes} \B(\ell^2)$ and define $\omega_0=\alpha (\beta \otimes [e_{j,i}]_{j,i=1}^{\infty})\alpha'\in \M_*\widehat{\otimes} \B(\ell^2)_*=(\M\bar{\otimes}\B(\ell^2))_*$. Unwinding the definitions, one finds that $\langle \gamma-(\Phi_\lambda\otimes\id)(\gamma),\omega_0\rangle=\langle\id-\Phi_\lambda , \Omega\rangle$. Setting $x=\tfrac{\gamma}{\|\gamma\|},\omega=\tfrac{\omega_0}{\|\omega_0\|}$ finishes the proof, as $\|\gamma\| \cdot \|\omega_0\|\le \tfrac{\eps'}{\eps}$, so that $|\la x-(\Phi_\lambda\otimes \id)(x),\omega\ra| \geq \tfrac{\eps}{\eps'} |\langle\id-\Phi_\lambda , \Omega\rangle|$ for each $\lambda \in \Lambda$.
\end{proof}

\begin{prop}\label{prop:someepsilon}
Let $\M$ be a non-injective von Neumann algebra. Then $\M$ has the $\eps$-separation property for some $0<\eps<1$.
\end{prop}

\begin{proof}
Assume by contradiction that $\M$ does not have $\eps$-separation property for any $0<\eps<1$. That is, for all $0<\eps<1$ there is a net $(\Phi_{\eps,\lambda})_{\lambda\in \Lambda_\eps}$ of normal, finite rank UCP maps such that
\begin{equation}\label{eq22}
\limsup_{\lambda\in \Lambda_\eps} \bigl|\bigl\la x-\Phi_{\eps,\lambda}(x),\omega\bigr\ra\bigr|\le \eps
\end{equation}
for all $x\in \M,\omega\in \M_*$ with $\|x\|=\|\omega\|=1$. Next we construct a new net of normal, finite rank, UCP maps on $\M$, $(\Psi_{F,G,\eps})_{(F,G,\eps)}$ where $F\subset \M,G\subset \M_*$ are finite non-empty sets and $0<\eps<1$. We declare $(F,G,\eps)\le (F',G',\eps')$ if and only if $F\subset F', G\subset G'$ and $\eps \ge \eps'$. For such a triple we choose $\Psi_{F,G,\eps}=\Phi_{\eps,\lambda}$ with $\lambda\in \Lambda_\eps$ such that
\[
\bigl|\bigl\la x-\Psi_{F,G,\eps}(x),\omega\bigr\ra\bigr|=
\bigl|\bigl\la x-\Phi_{\eps,\lambda}(x),\omega\bigr\ra\bigr|\le 2\eps \|x\|\|\omega\| \quad(x\in F, \omega\in G).
\]
The index $\lambda$ as above exists due to \eqref{eq22}. Now it is easy to see that the net $(\Psi_{F,G,\eps})_{(F,G,\eps)}$ implements the $w^*$-CPAP of $\M$, which gives us a contradiction.
\end{proof}

In view of the above the focus will from now on be on finding for a non-injective von Neumann algebra an explicit set of $\eps\in (0,1)$ for which the $\eps$-separation property holds. We do not know whether the $\eps$-separation property in fact depends on $\eps$; see the discussion in the end of the paper.

For a compact quantum group $\GG$, let us denote $\bf{N}_{\GG}=\sup_{\alpha\in \Irr(\GG)} \dim(\alpha)\in \NN\cup\{+\infty\}$. 

\begin{tw}\label{thm2}
	Let $\GG$ be a compact quantum group such that $\bf{N}_{\GG}<+\infty$. If $\GG$ is not coamenable, then $\LL^{\infty}(\GG)$ has the $\eps$-separation property for all $0<\eps<\tfrac{1}{\bf{N}_{\GG}}$.
\end{tw}

\begin{proof}
	By \cite[Theorem 4.3]{KS} we know that $\GG$ is of Kac type. Since $\GG$ is of Kac type, with any normal, finite rank UCP map $\Phi\colon \LL^{\infty}(\GG)\rightarrow \LL^{\infty}(\GG)$ we can associate a quantum Herz-Schur multiplier $\wt{\Phi}$ in a canonical way. As this construction is quite well-known, we will only present the relevant formulas. Consider the normal UCP map (\cite[Section 7.1]{Brannan}) 
	\begin{equation}\label{eq13}
		\Delta^\sharp\colon \LL^{\infty}(\GG)\bar{\otimes}\LL^{\infty}(\GG)\ni U^{\alpha}_{i,j}\otimes U^{\beta}_{k,l}\mapsto
		\delta_{\alpha,\beta}\delta_{j,k} \tfrac{U^{\alpha}_{i,l}}{\dim(\alpha)}\in \LL^{\infty}(\GG);
	\end{equation}
	in particular $\Delta^\sharp \Delta=\id$. Now define
	\[
	\wt{\Phi}=\Delta^\sharp(\Phi\otimes \id)\Delta\colon \LL^{\infty}(\GG)\rightarrow \LL^{\infty}(\GG).
	\]
	$\wt{\Phi}$ is a normal UCP map and $\wt{\Phi}=\Theta^l(a)$ for some $a\in \A(\whG)$ -- for the proof, see \cite[Section 6.3.2]{Brannan}.
	
	Take $0<\eps<\tfrac{1}{\bf{N}_{\GG}}$. Assume by contradiction that the claim does not hold, i.e.~there is a net $(\Phi_\lambda)_{\lambda\in \Lambda}$ of normal, UCP, finite rank maps $\LL^{\infty}(\GG)\rightarrow\LL^{\infty}(\GG)$ such that for all $x\in \LL^{\infty}(\GG),\omega\in \LL^1(\GG)$ with $\|x\|=\|\omega\|=1$ we have
	\begin{equation}\label{eq9}
		\limsup_{\lambda\in \Lambda} |\la x-\Phi_\lambda(x),\omega\ra |\le \eps.
	\end{equation}
	For each $\lambda \in \Lambda$  set
	\[
	\wt{\Phi}_\lambda=\Delta^\sharp( \Phi_\lambda\otimes \id)\Delta=\Theta^l(a_\lambda)
	\]
	for some multipliers $a_\lambda \in \A(\whG)$. We can write them as $a_\lambda=(\omega_\lambda\otimes \id)(\ww^{\GG })=(\id\otimes \omega_\lambda)(\ww^{\whG *})$ for vector states $\omega_\lambda=\omega_{\xi_\lambda}\in \LL^1(\GG)$, then $\wt{\Phi}_\lambda=(\omega_\lambda\otimes \id)\Delta$. For $n\in\mathbb{N}$ choose norm $1$ vectors $\xi_{\lambda,n}\in \Lambda_h(\operatorname{Pol}(\GG))$ such that $\|\xi_\lambda-\xi_{\lambda,n}\|\le \tfrac{1}{n}$ and let $\wt{\Phi}_{\lambda,n}=(\omega_{\xi_{\lambda,n}}\otimes \id)\Delta$. These are normal UCP quantum Herz-Schur multipliers. Furthermore, $\wt{\Phi}_{\lambda,n}$ are finite rank since $\xi_{\lambda,n}\in \Lambda_h(\operatorname{Pol}(\GG))$. We will obtain a contradiction with Proposition \ref{prop5}. For this, we need to show that for all $x\in \Pol(\GG),\omega\in\LL^1(\GG)$ with $\|x\|_2=\|\omega\|=1$ we have
	\begin{equation}\label{eq8}
		\limsup_{(\lambda,n)\in \Lambda\times \NN}
		|\la x-\wt{\Phi}_{\lambda,n}(x ),\omega \ra | \le \bf{N}_{\GG}\eps
	\end{equation}
	(note that ${\bf{N}}_{\GG}\eps<1$). However, an inspection of the proof shows that it is enough to prove \eqref{eq8} for $x=\tfrac{U^{\alpha}_{\xi,\eta}}{\|U^{\alpha}_{\xi,\eta}\|_2},\omega=\tfrac{h(U^{\alpha *}_{\zeta,\eta} \cdot )}{\|h(U^{\alpha *}_{\zeta,\eta} \cdot) \|}$ and fixed $\alpha\in \Irr(\GG),\xi,\eta,\zeta\in \msf{H}_\alpha\setminus\{0\}$. Indeed, the argument leading to \eqref{eq11} was the only place in the proof of Proposition \ref{prop5} where the assumption was used. Thus let us show \eqref{eq8} for this pair $x,\omega$. Fix $(\lambda,n)\in \Lambda\times \NN$ and note that
\begin{align*}
		|\la x - \wt{\Phi}_{\lambda,n}(x), \omega \ra| \le &
		|\la ((\omega_{\xi_\lambda}-\omega_{\xi_{\lambda,n}})\otimes \id)\Delta(x), \omega  \ra |+|
		\la  x - \wt{\Phi}_{\lambda}(x),\omega \ra |\\
		&\le 
		\tfrac{2\|x\|}{n} + 
		|\la  (  (\id-\Phi_\lambda)\otimes \id)\Delta(x),\omega\circ\Delta^\sharp\ra |
\end{align*}
Further 
\begin{align*}
	|\la  (  (\id&-\Phi_\lambda)\otimes \id)\Delta(x),\omega\circ\Delta^\sharp\ra | \\&= 
		\tfrac{1}{\|U^{\alpha}_{\xi,\eta}\|_2 \|h(U^{\alpha *}_{\zeta,\eta} \cdot )\|}
		\bigl|\sum_{m=1}^{\dim(\alpha)}
		\la 
		(U^{\alpha}_{\xi,\xi^\alpha_m}-
		\Phi_\lambda(U^{\alpha}_{\xi,\xi^\alpha_m}))
		\otimes  U^{\alpha}_{\xi^\alpha_m,\eta},
		h(U^{\alpha *}_{\zeta,\eta}\cdot )\circ\Delta^\sharp \ra \bigr|\\
		&= 
		\tfrac{1}{\|U^{\alpha}_{\xi,\eta}\|_2 \|h(U^{\alpha *}_{\zeta,\eta} \cdot )\|}
		\bigl|\sum_{m=1}^{\dim(\alpha)}
		\la (U^{\alpha}_{\xi,\xi^\alpha_m}
		-
		\sum_{a,b=1}^{\dim(\alpha)}\dim(\alpha) h( U^{\alpha *}_{\xi^\alpha_a,
			\xi^\alpha_b}\Phi_\lambda(U^{\alpha}_{\xi,\xi^{\alpha}_m}))U^{\alpha}_{\xi^\alpha_a,\xi^\alpha_b})
		\otimes U^{\alpha}_{\xi^\alpha_m,\eta},\\
		&\quad\quad\quad\quad\quad\quad\quad\quad
		\quad\quad\quad\quad\quad\quad\quad\quad
		\quad\quad\quad\quad\quad\quad\quad\quad
		\quad\quad\quad\quad\quad \quad
		h(U^{\alpha *}_{\zeta,\eta}\cdot )\circ\Delta^\sharp
		\ra \bigr|
		\\&=
		\tfrac{1}{\|U^{\alpha}_{\xi,\eta}\|_2 \|h(U^{\alpha *}_{\zeta,\eta} \cdot )\|}
		\bigl|\sum_{m=1}^{\dim(\alpha)}\tfrac{1}{\dim(\alpha)}
		\la 
		U^{\alpha}_{\xi,\eta}
		-\sum_{a=1}^{\dim(\alpha)}\dim(\alpha) h( U^{\alpha *}_{\xi^\alpha_a,
			\xi^\alpha_m}\Phi_\lambda(U^{\alpha}_{\xi,\xi^{\alpha}_{m}}))U^{\alpha}_{\xi^{\alpha}_a,\eta},
		h(U^{\alpha *}_{\zeta,\eta}\cdot ) \ra \bigr| 
		\\&=
		\tfrac{1}{\|U^{\alpha}_{\xi,\eta}\|_2 \|h(U^{\alpha *}_{\zeta,\eta} \cdot )\|}
		\bigl|
		\tfrac{\la \xi|\zeta\ra \|\eta\|^2}{\dim(\alpha)}
		-\sum_{m,a=1}^{\dim(\alpha)} h( U^{\alpha *}_{\xi^\alpha_a,
			\xi^\alpha_m}\Phi_\lambda(U^{\alpha}_{\xi,\xi^{\alpha}_{m}}))
		\tfrac{\la \xi^{\alpha}_a|\zeta\ra \| \eta\|^2 }{\dim(\alpha)}\bigr|
		\\	&=
		\tfrac{1}{\|U^{\alpha}_{\xi,\eta}\|_2 \|h(U^{\alpha *}_{\zeta,\eta} \cdot )\|}
		\bigl|
		\tfrac{\la \xi|\zeta\ra \|\eta\|^2}{\dim(\alpha)}
		-\sum_{m=1}^{\dim(\alpha)} h( U^{\alpha *}_{\zeta,\xi^\alpha_m}\Phi_\lambda(U^{\alpha}_{\xi,\xi^\alpha_m}))
		\tfrac{\|\eta\|^2  }{\dim(\alpha)} \bigr|
		\\&=
\tfrac{\|\eta\|^2}{\|\xi\|\|\eta\| \|h(U^{\alpha *}_{\zeta,\eta} \cdot )\|\dim(\alpha)^{1/2}}
		\bigl|
		\la \xi|\zeta\ra -\sum_{m=1}^{\dim(\alpha)} h( U^{\alpha *}_{\zeta,
			\xi^{\alpha}_{m}}\Phi_\lambda(U^{\alpha}_{\xi,\xi^\alpha_m})) \bigr|.
	\end{align*}
Combining the two formulas displayed above we obtain
	\[\begin{split}
		|\la x - \wt{\Phi}_{\lambda,n}(x),\omega \ra| &\le 
		\tfrac{2\|x\|}{n} + 
		\tfrac{\|\eta\|}{ \|\xi\| \|h(U^{\alpha *}_{\zeta,\eta} \cdot )\|\dim(\alpha)^{1/2}}
		\sum_{m=1}^{\dim(\alpha)}
		\bigl| \tfrac{\la \xi|\zeta\ra }{\dim(\alpha)} - 
		\la 
		\Phi_\lambda(U^{\alpha}_{\xi,\xi^{\alpha}_{m}}),
		h(U^{\alpha *}_{\zeta,\xi^\alpha_m} \cdot )  \ra \bigr|\\
		&=
		\tfrac{2\|x\|}{n} + 
		\tfrac{\|\eta\|}{ \|\xi\| \|h(U^{\alpha *}_{\zeta,\eta} \cdot )\|\dim(\alpha)^{1/2}}
		\sum_{m=1}^{\dim(\alpha)}
		\bigl| \la 
		U^{\alpha}_{\xi,\xi^\alpha_m}-
		\Phi_\lambda(U^{\alpha}_{\xi,\xi^\alpha_m}),
		h(U^{\alpha *}_{\zeta,\xi^\alpha_m} \cdot )  \ra \bigr|.
	\end{split}\]
	Consequently by \eqref{eq9}
	\begin{equation}\label{eq12}
		\limsup_{(\lambda,n)\in \Lambda\times \NN}
		|\la x - \wt{\Phi}_{\lambda,n}(x),\omega\ra| \le 
		\tfrac{\|\eta\|}{ \|\xi\| \|h(U^{\alpha *}_{\zeta,\eta} \cdot )\|\dim(\alpha)^{1/2}}
		\sum_{m=1}^{\dim(\alpha)}
		\|U^{\alpha}_{\xi,\xi^\alpha_m }\|
		\| h(U^{\alpha *}_{\zeta,\xi^\alpha_m} \cdot)\|
		\eps.
	\end{equation}
	Now we use the assumption $\bf{N}_{\GG}<+\infty$ to obtain a lower bound on the norm of functional $h(U^{\alpha *}_{\zeta,\eta}\cdot)$. Since
	\[
	\|h(U^{\alpha *}_{\zeta,\eta}\cdot)\|\ge \tfrac{
		|h(U^{\alpha *}_{\zeta,\eta} U^{\alpha}_{\zeta,\eta})|}{
		\|U^{\alpha }_{\zeta,\eta}\|}=\tfrac{ \|\zeta\|^2 \|\eta\|^2 }{\dim(\alpha) \|U^{\alpha}_{\zeta,\eta}\|},\]
	in our situation we have 
	\[
	\tfrac{1}{\|h(U^{\alpha *}_{\zeta,\eta}\cdot)\|}\le
	\tfrac{\dim(\alpha) \|U^{\alpha}_{\zeta,\eta}\|}{\|\zeta\|^2 \|\eta\|^2}\le 
	\tfrac{\bf{N}_{\GG}\|\zeta\|\,\|\eta\|}{\|\zeta\|^2 \|\eta\|^2}=
	\tfrac{\bf{N}_{\GG}}{\|\zeta\|\, \|\eta\|}.
	\]
	Combining this with inequality \eqref{eq12} and standard inequalities $\|U^{\alpha}_{\xi,\xi^\alpha_m}\|\le \|\xi\|$, $ \|h(U^{\alpha *}_{\zeta,\xi^\alpha_m}\cdot)\|\le \|U^\alpha_{\zeta,\xi^\alpha_m } \|_2=\tfrac{\|\zeta\|}{\dim(\alpha)^{1/2}}$ we get
	\[
	\limsup_{(\lambda,n)\in \Lambda\times \NN}
	|\la  x - \wt{\Phi}_{\lambda,n}(x),\omega\ra| \le 
	\tfrac{\|\eta\|\bf{N}_{\GG}}{
		\|\xi\|\,\|\zeta\|\,\|\eta\|
		\dim(\alpha)^{1/2}}
	\sum_{m=1}^{\dim(\alpha)}\|\xi\|
	\tfrac{\|\zeta\|}{\dim(\alpha)^{1/2}}
	\eps=\bf{N}_{\GG}\eps
	\]
	which shows \eqref{eq8} and ends the proof.
\end{proof}

Theorem \ref{thm2} in particular applies to $\GG=\widehat{\Gamma}$, where $\Gamma$ is a discrete group; this is Corollary \ref{corC} of the introduction.

%\begin{cor}\label{cor:classicalsep}
%	Let $\Gamma$ be a discrete group, $\eps \in (0,1)$. Then $\operatorname{vN}(\Gamma)$ has the $\eps$-separation property if and only if %$\operatorname{vN}(\Gamma)$ is non-injective.
%\end{cor}

\smallbreak

In the last result of this section we show that  if we consider rather the matrix $\eps$-separation property, then we can drop the assumptions on the boundedness of the dimension of irreducible representations of $\GG$ and obtain the result valid for all $\eps \in (0,1)$. 

\begin{tw}\label{thm:matrix}
	Let $\GG$ be a compact quantum group of Kac type which is not coamenable. Then  $\LL^{\infty}(\GG)$ has the matrix $\eps$-separation property for all $\eps \in (0,1)$.
\end{tw}

\begin{proof}
	The argument will follow the lines of the proof of Theorem \ref{thm2}.
	Assume by contradiction that there is $0<\eps<1$ and a net $(\Phi_\lambda)_{\lambda\in \Lambda}$ of normal, finite rank, UCP maps on $\LL^{\infty}(\GG)$ such that
	\[
	\limsup_{\lambda\in \Lambda} |\la x-(\Phi_\lambda\otimes \id)(x),\omega \ra |\le \eps
	\]
	for all $n\in\NN$ and $x\in \LL^{\infty}(\GG)\otimes \M_n,\omega\in (\LL^{\infty}(\GG)\otimes \M_n)_*$ with $\|x\|=\|\omega\|=1$. Thus in particular for any $\alpha\in \Irr(\GG)$ and $x=U^{\alpha}, \omega=(h\otimes \operatorname{tr})(U^{\alpha *}\cdot)$ we have
	\begin{equation}\label{eq14}
		\limsup_{\lambda\in \Lambda} \,
		\bigl|(h\otimes \operatorname{tr})\bigl( U^{\alpha *} (U^{\alpha}-(\Phi_{\lambda}\otimes \id)(U^{\alpha}))\bigr)\bigr|\le \eps
	\end{equation}
	(where $\operatorname{tr}$ is the normalised trace on $\M_n$). Recall the map $\Delta^{\sharp}$ of \eqref{eq13} and define again for each $\lambda \in \Lambda$
	\[
	\wt{\Phi}_{\lambda}=\Delta^{\sharp}
	(\Phi_{\lambda} \otimes \id)\Delta.
	\]
Using the fact that $\Delta^{\sharp}=\Delta^{-1}\mathbb{E}$, where $\mathbb{E}\colon \LL^{\infty}(\QG)\bar{\otimes}\LL^{\infty}(\QG)\rightarrow \Delta(\LL^{\infty}(\QG))$ is the normal $h\otimes h$ -preserving conditional expectation, we compute as follows:
\[\begin{split}
(h\otimes \oon{tr})( U^{\alpha *} (\wt{\Phi}_\lambda\otimes \id)(U^{\alpha}))&=
(h\otimes h\otimes \oon{tr})(\Delta\otimes \id)\bigl( U^{\alpha *} (\wt{\Phi}_\lambda\otimes \id)(U^{\alpha})\bigr)\\
&=
(h\otimes h\otimes \oon{tr})\bigl(
(\Delta\otimes \id)(U^{\alpha *})
(\mathbb{E}\otimes \id)\bigl( (\Phi_\lambda\otimes \id)(U^{\alpha})_{13} U_{23}^{\alpha}\bigr)\bigr)\\
&=
(h\otimes h\otimes \oon{tr})\bigl(
(\mathbb{E}\otimes \id)\bigl(
(\Delta\otimes \id)(U^{\alpha *})
 (\Phi_\lambda\otimes \id)(U^{\alpha})_{13} U_{23}^{\alpha}\bigr)\bigr)\\
 &=
(h\otimes h\otimes \oon{tr})\bigl(
U^{\alpha * }_{23}U^{\alpha *}_{13}
 (\Phi_\lambda\otimes \id)(U^{\alpha})_{13} U_{23}^{\alpha}\bigr)
 \\&=
 (h\otimes \oon{tr})(U^{\alpha *} (\Phi_\lambda\otimes \id)(U^{\alpha})).
\end{split}\]
In the above calculation we have used the Tomiyama's theorem and the assumption that $h$ is tracial. Consequently, from \eqref{eq14} we obtain
\begin{equation}\label{eq23}
		\limsup_{\lambda\in \Lambda} \,
		\bigl|(h\otimes \operatorname{tr})\bigl( U^{\alpha *} (U^{\alpha}-(\wt{\Phi}_{\lambda}\otimes \id)(U^{\alpha}))\bigr)\bigr|\le \eps.
\end{equation}
Next, let us express explicitly the second factor appearing above:
\[\begin{split}
		(h\otimes \operatorname{tr})\bigl( U^{\alpha *}  (\wt{\Phi}_{\lambda}\otimes \id)(U^{\alpha})\bigr)&=
		\sum_{i,j,k,l=1}^{\dim(\alpha)}
		(h\otimes \operatorname{tr})\bigl(
		U^{\alpha *}_{i,j} \wt{\Phi}_{\lambda}(U^{\alpha}_{k,l})
		\otimes e^{\alpha}_{j,i} e^{\alpha}_{k,l}
		\bigr)\\
		&=
		\tfrac{1}{\dim(\alpha)}\sum_{i,j=1}^{\dim(\alpha)}
		h\bigl(U^{\alpha *}_{i,j}  \wt{\Phi}_{\lambda}(U^{\alpha}_{i,j})\bigr).
\end{split}\]
We thus obtain from \eqref{eq23} that 
	\[
	\limsup_{\lambda\in \Lambda}\,
	\bigl|1-
	\tfrac{1}{\dim(\alpha)}\sum_{i,j=1}^{\dim(\alpha)}
	h(U^{\alpha *}_{i,j}   \wt{\Phi}_{\lambda}(U^{\alpha}_{i,j}))\bigr|\le \eps.
	\]
	Each $\wt{\Phi}_{\lambda}$ is a normal UCP quantum Herz-Schur multiplier associated with a function $a_\lambda\in \A(\whG)$, again by \cite[Section 6.3.2]{Brannan}. As in the proof of Theorem \ref{thm2}, we can approximate each $\wt{\Phi}_{\lambda}$ (in a completely bounded norm) by a sequence $(\wt{\Phi}_{\lambda, m})_{m=1}^\infty$ of normal UCP quantum Herz-Schur multipliers associated with positive-definite functions $a_{\lambda,m} \in \mrm{c}_{c}(\whG)$. Increasing $\eps$ to $\eps' \in (0,1)$ we obtain then again for each $\alpha\in \Irr(\GG)$
	\begin{equation}\label{eq15}
		\limsup_{(\lambda, m)\in \Lambda\times \NN}\,
		\bigl|1-
		\tfrac{1}{\dim(\alpha)}\sum_{i,j=1}^{\dim(\alpha)}
		h(U^{\alpha *}_{i,j}   \wt{\Phi}_{\lambda,m}(U^{\alpha}_{i,j}))\bigr|\le \eps'.
	\end{equation}
	Note that we have for each $\alpha\in \Irr(\GG)$, $1 \le i,j \le \dim (\alpha), \lambda \in \Lambda, m \in \NN$ 
	\begin{equation} \label{eq16} \wt{\Phi}_{\lambda,m}(U^{\alpha}_{i,j})   = \sum_{k=1}^{\dim(\alpha)} (a_{\lambda,m})^{\alpha}_{i,k} U^{\alpha}_{k,j},\end{equation}
	so that \eqref{eq15} simplifies to
	\begin{equation}\label{eq17}
		\limsup_{(\lambda, m)\in \Lambda\times \NN}\,
		\bigl|1-
		\tfrac{1}{\dim(\alpha)}\sum_{k=1}^{\dim(\alpha)}
		(a_{\lambda,m})^{\alpha}_{k,k}\bigr|\le \eps'.
	\end{equation}
	
	Finally we average once again (from the other side), setting for each $\lambda \in \Lambda, m \in \NN$ 
	\[ \hat{\Phi}_{\lambda,m} = \Delta^{\sharp}
	(\id \otimes \wt{\Phi}_{\lambda,m})\Delta,\]
c.f.~\cite[Proposition 6.8]{DKV}. Clearly $\hat{\Phi}_{\lambda,m}$ is a normal UCP map.  A direct calculation using \eqref{eq16} shows that we have for each $\alpha\in \Irr(\GG)$, $1 \le i,j \le  \dim (\alpha)$ 
	\begin{align*} \hat{\Phi}_{\lambda,m} (U^{\alpha}_{i,j})&= \Delta^\sharp \bigl(\sum_{k=1}^{\dim(\alpha)}  U^{\alpha}_{i,k} \otimes \wt{\Phi}_{\lambda,m}(U^{\alpha}_{k,j})    \bigr) 
		= \Delta^\sharp \bigl(\sum_{k=1}^{\dim(\alpha)}  U^{\alpha}_{i,k} \otimes \sum_{l=1}^{\dim(\alpha)} (a_{\lambda,m})^{\alpha}_{k,l}  U^{\alpha}_{l,j}\bigr)
		\\&= \tfrac{1}{\dim(\alpha)}\sum_{k=1}^{\dim(\alpha)} (a_{\lambda,m})^{\alpha}_{k,k}  U^{\alpha}_{i,j}.
	\end{align*} 
	Thus each $\hat{\Phi}_{\lambda,m}$ is a central quantum Herz-Schur multiplier, associated to a positive-definite function $b_{\lambda, m} \in \mc{Z} \mrm{c}_{c}(\whG)$ given by 
	\[ b_{\lambda, m} =  \sum_{\alpha\in\Irr(\whG)}\bigl(\tfrac{1}{\dim(\alpha)}\sum_{k=1}^{\dim(\alpha)}\bigr (a_{\lambda,m})^{\alpha}_{k,k}\bigr) p_\alpha.\]
	Finally define $c_{\lambda,m} = \frac{1}{2} (b_{\lambda, m} + b_{\lambda, m}^*) \in \mc{Z}\mrm{c}_{c}(\whG)$. This is again a central, real valued positive-definite function on $\whG$, associated to the quantum Herz-Schur multiplier $\tfrac{1}{2}(\hat{\Phi}_{\lambda,m} + \hat{\Phi}_{\lambda,m}\circ R_{\GG})$. Inequality \eqref{eq17} implies that 
	\[ \limsup_{(\lambda, m)\in \Lambda\times \NN}\,
	(1-c_{\lambda,m}^\alpha)\le \eps',\]
	which in turn shows the coamenability of $\GG$ via Corollary \ref{cor3} (2).
\end{proof}

Note  that Theorem \ref{thm:matrix} can be reformulated in the spirit of Corollary \ref{corC}, leading to Theorem \ref{thmB} from the introduction.

%\begin{cor}\label{cor:Kacmatrix}
%\noindent \emph{Proof of Theorem }\ref{thmB}:	

%	Let $\GG$ be a compact quantum group of Kac type, $\eps \in (0,1)$. Then $\LL^\infty(\GG)$ has the matrix $\eps$-separation property if and only if $\LL^\infty(\GG)$ is non-injective.
%	\qed
%\end{cor}

\begin{proof}[Proof of Theorem \ref{thmB}]
	This is an immediate consequence of Remark \ref{rem:sep}, Theorem \ref{thm:matrix} and the fact that $\GG$ is coamenable if and only if  $\LL^\infty(\GG)$ is injective, which goes back to \cite{Ruan} (see also \cite{Brannan}).
\end{proof}

The analogue of Theorem \ref{thm:matrix} holds also for the Haagerup property. We formalise it in the next theorem.

\begin{tw}\label{thm:Haagerup}
	Let $\GG$ be a compact quantum group of Kac type  and let $\M = \LL^\infty(\GG)$. Fix $\eps \in (0,1)$. Suppose that there exists a net	$(\Phi_\lambda)_{\lambda\in \Lambda}$  of normal, $h$-preserving UCP maps on $\M$ which have compact $\LL^2$-implementations (in the sense of \cite{Jol}), such that for  every $n \in \bn$, $x\in \M\otimes \M_n,\omega\in (\M\otimes \M_n )_*$ with $\|x\|=\|\omega\|=1$ we have $\limsup_{\lambda\in \Lambda} |\la x-(\Phi_\lambda\ot \id)(x),\omega\ra|\leq \eps$. Then $\whG$ has the Haagerup property.
\end{tw}
\begin{proof}
	As the proof is very similar to that above (and in fact even simpler), we only outline the main steps.	We again begin by passing from the net $(\Phi_\lambda)_{\lambda\in \Lambda}$ to $(\wt{\Phi}_\lambda)_{\lambda\in \Lambda}$ by averaging; the proof of \cite[Theorem 7.7]{QHap} shows that the positive-definite functions associated to the resulting multipliers belong to $\mrm{c}_{0}(\whG)$. Averaging again and then symmetrizing gives a net of normalized real-valued positive-definite functions which in the limit are `$\eps$-close to $\I$'. This contradicts Corollary \ref{cor2}.
	
\end{proof}

We finish the paper by stating two natural open questions. 

\begin{quest}
Let $\M$ be a von Neumann algebra and  let $\eps \in (0,1)$. 
\begin{itemize}
	\item Is the $\eps$-separation property of $\M$ equivalent to the matrix $\eps$-separation property of $\M$?
	\item Is the matrix $\eps$-separation property of $\M$ equivalent to non-injectivity of $\M$?
\end{itemize}
\end{quest}
Our results show that the answer to the first question above is positive for  von Neumann algebras of discrete groups and answer to the second question is positive for  von Neumann algebras of unimodular discrete quantum  groups. Note that in view of Proposition \ref{prop:someepsilon} the positive answers to questions above are equivalent to the negative answer to the following question.

\begin{quest}
	Let $\M$ be a von Neumann algebra. 
	\begin{itemize}
		\item Does the $\eps$-separation property of $\M$ depend on $\eps\in (0,1)$?
	\end{itemize}
\end{quest}

One could of course formulate variants of the above questions relevant for the von Neumann algebraic Haagerup property.

Finally one can ask whether similar separation concepts can be introduced in the context of \emph{weak amenability} or \emph{completely bounded approximation property} (\cite{bo}, \cite{Brannan}). Here however the questions become significantly more difficult, even for classical groups. The reason is that we no longer can exploit positivity of the relevant approximants and there is no obvious replacement for tools such as the Godement mean used in this paper.    

\smallskip

\noindent {\bf Acknowledgments. } 
J.K.\ was partially supported by EPSRC grants EP/T03064X/1 and EP/T030992/1. A.S.\ was  partially supported by the National Science Center (NCN) grant no. 2020/39/I/ST1/01566. A.S.\ would also like to express the gratitude to Matthew Daws and Christian Voigt for making possible his visit to Glasgow in May 2023, where some of the work on this paper was done. Finally we thank the referee for useful comments which improved the presentation.

\end{document}